\documentclass[11pt]{article}
\usepackage[utf8]{inputenc}
\usepackage[T1]{fontenc}
\usepackage{lmodern}
\usepackage{geometry}
\author{David Forsman\\
Université catholique de Louvain\\
\texttt{david.forsman@uclouvain.be}}
\title{On the Multicategorical Meta-Theorem and\\
the Completeness of Restricted Algebraic Deduction Systems}
\date{\today}

\usepackage{graphicx}
\usepackage{amsthm, amssymb, mathtools}
\usepackage{enumitem}
\usepackage{multicol}
\usepackage{hyperref}

\usepackage{tikz-cd}
\usepackage{caption, subcaption}
\usepackage[numbers]{natbib}
\usepackage{url}
\usepackage{microtype}

  \newtheorem*{theorema}{Theorem}

\theoremstyle{plain}
\newtheorem{theorem}{Theorem}[section]
\newtheorem{lemma}[theorem]{Lemma}
\newtheorem{corollary}[theorem]{Corollary}

\theoremstyle{definition}
\newtheorem{definition}[theorem]{Definition}
\newtheorem{example}[theorem]{Example}

\numberwithin{equation}{section} 

\newcommand{\N}{\mathbb{N}}

\newcommand{\ra}{\rightarrow}

\begin{document}

\maketitle
\begin{abstract}\noindent
Eight categorical soundness and completeness theorems are established within the framework of algebraic theories. Exactly six of the eight deduction systems exhibit complete semantics within the cartesian monoidal category of sets. The multicategorical meta-theorem via soundness and completeness enables the transference of properties of families of models from the cartesian monoidal category of sets to $\Delta$-multicategories $C$.

A bijective correspondence $R \mapsto \Delta_R$ is made between context structures $R$ and structure categories $\Delta$, which are wide subcategories of $\textbf{FinOrd}$ consisting of finite ordinals and functions. Given a multisorted signature $\sigma$ with a context structure $R$, an equational deduction system $\vdash_R$ is constructed for $R$-theories. The models within $\Delta_R$-multicategories provide a natural semantic framework for the deduction system $\vdash_R$ for modelable context structures $R$. Each of the eight modelable context structures $R$ is linked with a soundness and completeness theorem for the deduction system $\vdash_R$.
\end{abstract}

\tableofcontents
\section{Introduction}
In logic there is a blurred distinction between syntax and semantics. Often one has in mind a structure or a family of structures about which one would like to know more. Mathematical logic studies this phenomenon through the concepts of language and the structures to which language refers, called syntax and semantics respectively. Specifically, syntax consists of an alphabet, propositions, and the associated deduction system for propositions. Traditionally, semantics refers to the models consisting of sets, functions and relations compatible with a given syntax. Soundness is the property that the deduction rules of the syntax are compatible with the associated models. Completeness then describes the maximality of the deduction system in the sense that every property of all models that can be expressed in the language is obtained by deduction.

Soundness and completeness theorems allow us to study the class of models of a theory by studying only the syntax of the theory. There is a meaningful consequence of the soundness and completeness of a given syntax. Namely, if the soundness of a deduction system holds for some collection $A$ of models, and the completeness of the deduction holds for some family $B$ of models, then one can transfer properties from the class $B$ to the class $A$. This creates a path for results to pass from $B$ to $A$. We call this notion a meta-theorem.

The Eckmann-Hilton argument establishes a condition under which two unital magma structures on the same set form a single commutative monoid. This argument extends from the monoidal category of sets to any symmetric monoidal category. An intriguing question arises: does this generalization follow simply because the claim holds in the context of sets? The answer is affirmative. For any equational linear theory $E\cup\{\phi\}$ we demonstrate $E\vDash_{\textbf{Set}}\phi$ implies $E\vDash_C \phi$ for all symmetric multicategories $C$. This result is further generalized to $\Delta$–multicategories across six distinct structure categories $\Delta$.

We introduce four main concepts in this paper:
\begin{multicols}{2}
\begin{enumerate}
    \item \textbf{Context Structure} $R$ dictates when a sequence of variables $c$ forms a context for another sequence $v$ via the relation $cRv$.
    \item \textbf{Equational Deduction System} $\vdash_R$ defined by an algebraic signature $\sigma$ equipped with a context structure $R$.
    \item \textbf{Bijection} $R\mapsto \Delta_R$ maps context structures $R$ to \textbf{structure categories} $\Delta$, where a structure category is a suitable wide subcategory of $\textbf{FinOrd}$ encompassing finite ordinals and functions.
    \item $\Delta$-\textbf{multicategory} $C$ is endowed with a specified $\Delta$–action on its hom-sets.
\end{enumerate}
\end{multicols}
While there are infinitely many context structures $R$, our focus centers on eight specific structures called modelable context structures. These structures enable $R$-theories to have models within $\Delta_R$-multicategories. We establish the soundness and completeness of $\vdash_R$ for each modelable $R$ with respect to categorical semantics. Notably, six out of these eight modelable context structures achieve complete semantics just in sets. The Multicategorical Meta-Theorem emerges as a corollary:

\begin{theorema}[Multicategorical Meta-Theorem]
Let $\sigma$ be a signature with a cartesian or balanced and modelable context structure $R$. Let $E\cup\{\phi\}$ be an $R$-theory. Let $C$ be a $\Delta_R$-multicategory.  Then
$$
E\vDash_\textbf{Set}\phi \text{ implies }E\vDash_C \phi.
$$
\end{theorema}

The remainder of this paper is structured as follows: In Section 2, we define the notion of a structure category $\Delta$ as a wide subcategory of $\textbf{FinOrd}$ consisting of functions between finite ordinals. A structure category $\Delta$ enables the definition of $\Delta$–multicategories through its action on the homsets of a multicategory. Section 3 introduces the basic concepts in universal algebra, divided into subsections covering syntax, soundness, and completeness within the context of sets. Syntax encompasses signatures equipped with a context structure $R$, $R$-terms, and $R$-theories. We establish a bijective correspondence $R \mapsto \Delta_R$ between context structures and structure categories, identifying precisely eight modelable context structures $R$. The soundness subsection addresses the definition of models in $\Delta_R$-multicategories for $R$-theories, demonstrating soundness across all models in $\Delta_R$ categories. In the completeness subsection, we establish that six out of the eight deduction systems $\Delta_R$ exhibit complete semantics within the cartesian multicategory of sets, culminating in the Multicategorical Meta-Theorem, a central result of this paper.

Section 4 is dedicated to constructing the universal model for a $\sigma, R$-theory $E$ and proving the categorical completeness of $R$-deduction for a modelable $R$. We construct the universal $E, R$-model by developing a linear $\Sigma$-theory $\overline{E}$ that conservatively extends $E$. The initial $\overline{E}$-model serves as the universal model, providing a free multicategory equipped with a model for $E$. This construction is closely related to the functorial semantics proposed by Lawvere \cite{FunctorialSemantics}. A natural transformation $\eta \colon F \Rightarrow G \colon M \rightarrow N$ in the $2$-category of $\Delta$-multicategories can be interpreted as a homomorphism from the model $F$ to $G$ of a theory $M$ in the $\Delta$-multicategory $N$.

Section 5 concludes the paper by summarizing the main findings.

\section{Multicategories}
Structures like cartesian and symmetric multicategories have been extensively studied, as discussed in works such as \cite{Leinster_2004} and \cite{gould2010coherence}. We introduce $\Delta$–multicategories, a generalization encompassing multicategories, symmetric multicategories, and cartesian multicategories within a unified framework. Initially, we define structure categories $\Delta$ as wide subcategories of $\textbf{FinOrd}$ comprising finite ordinals and functions. We proceed to define the action of $\Delta$ on the homsets of a multicategory, leading to the formal definition of $\Delta$–multicategories. Special cases include instances where $\Delta$ comprises all functions, bijections, or identities, corresponding respectively to cartesian multicategories, symmetric multicategories, and multicategories. This sets the stage for a bijective correspondence between structure categories and context structures defined later.

\subsection{Structure Categories}
A structure category is a wide subcategory of $\textbf{FinOrd}$ of finite ordinal closed under coproduct of functions and components of similarities. We give infinitely many examples of structure categories using the notion of a structure monoid.
\begin{definition}[Finite Ordinals]
    We define the category $\textbf{FinOrd}$ of finite ordinals and functions as follows:
    \begin{itemize}
        \item Objects are sets $[n]\coloneqq \{m\in\N\mid 0<m\leq n\}$ for $n\in\N$, called finite ordinals.
        \item Morphism $[m]\rightarrow [n]$ is a function $[m]\rightarrow [n]$. The composition is the usual function composition.
    \end{itemize}
    We define the $k$-ary coproduct functor $+_k\colon \textbf{FinOrd}^k\rightarrow \textbf{FinOrd}$.
        \item We define $+_k([m_1],\ldots, [m_k]) = [m_1+\ldots+ m_k]$ for $m_1,\ldots, m_k\in\N$ and for functions $f_i\colon [m_i]\rightarrow [n_i], 0<i\leq k,$ between finite ordinals we set
        $$
        +_n(f_1,\ldots, f_k) = f_1+\ldots + f_k\colon [M_k]\rightarrow [N_k],M_{i-1}+x\mapsto N_{i-1}+f_i(x)
        $$
        for $0< x \leq m_i$ and $i\leq k$. Here $M_i = \Sigma_{j\leq i} m_i$ and $N_i = \Sigma_{j\leq i} n_i$ for $i\leq k$. 
\end{definition}
\begin{definition}[Similarity of a function]
    Let $\theta\colon [m]\rightarrow [n]$ be a function between finite ordinals. For a category $C$, we set $\theta^*\colon C^n\rightarrow C^m$, where $(x_1,\ldots, x_n)\mapsto (x_{\theta(1)},\ldots, x_{\theta(m)})$. We define a natural transformation $\theta'\colon +_m\circ\theta^*\Rightarrow +_n\colon \textbf{FinOrd}^n\rightarrow \textbf{FinOrd}$ by setting
    $$
    \theta'_{k_1,\ldots, k_m}\colon [L_m]\rightarrow [K_n]\text{, where } L_{i-1} + x\mapsto K_{\theta(i) - 1} + x
    $$
    for $k_1,\ldots, k_n\in\N$. Here we denote $L_i = k_{\theta(1)}+\ldots + k_{\theta(i)}$ and $K_j = k_1+\ldots + k_j$ for $i\leq m$ and $j\leq n$. In addition, we call the natural transformation $\theta'$ the similarity natural transformation of $\theta$. The function $\theta$ is recovered back via $\theta'_{1,\ldots, 1} = \theta$.
\end{definition}
\begin{definition}[Structure category]
    A wide (containing all objects) subcategory $\Delta$ of $\textbf{FinOrd}$ of finite ordinals is said to be a structure category if the following holds:
    \begin{itemize}
        \item  $\Delta$ is closed under the coproduct of morphisms.
        
        \item If a function $\theta\colon [m]\rightarrow [n]$ is in $\Delta$, then the component $\theta'_{k_1,\ldots, k_n}\colon [k_{\theta(1)}+\ldots + k_{\theta(m)}]\rightarrow [k_1+\ldots +k_n]$ of the similarity of $\theta$ is in $\Delta$ for $k_1,\ldots, k_n\in\N$. 
    \end{itemize}
    \end{definition}
    
    There are infinitely many structure categories. Certain submonoids of natural numbers with multiplicative structure allow us to define structure categories.
    \begin{definition}[Structure monoid]
        Let $I\subset\N$. We call $I$ a structure monoid, if the two conditions hold:
        \begin{itemize}
            \item $1\in I$.
            \item If $k,a_1,\ldots, a_k\in I$, then $a_1+\ldots + a_k\in I$. 
        \end{itemize} 
        Each structure monoid $I$ determines structure categories $\Delta_I,\Delta^I$ and when $0\not\in I$ structure categories $\Psi_I$ and $\Psi^I$:
        \begin{itemize}
            \item $\Delta^I$ consists of those functions $f\colon [m]\rightarrow [n]$ between finite ordinals where the cardinality of the fiber $f^{-1}\{i\}$ is in $I$ for each $i\in[n]$.
            \item $\Delta_I$ is the structure category generated by the functions $[n]\rightarrow [1],n\in I$.
            \item Assume $0\not\in I$. We set $\Psi_I$ to consist of those functions $f\colon[m]\rightarrow [n]$ in $\Delta^I$, where $\min(f^{-1}\{i\})\leq \min(f^{-1}\{j\})$ for all $i\leq j\leq n$. 
            \item Symmetrically, we set $\Psi^I$ similarly to $\Psi_I$ except we change taking of minimum with taking of maximum when $0\not\in I$.
        \end{itemize} 
    \end{definition}
    \noindent
    Notice that each structure monoid $I$ is a submonoid of $\N$ with the multiplicative structure. Consider $a,b\in I$. Now $ab = a_1+\ldots + a_b\in I$ where $a_i = a$ for $i\leq b$.

    \begin{lemma}\label{structure category lemma}
        Let $\Delta$ be a structure category. Define the set $I\subset \N$ where $I$ consists of all the cardinalities of the fibers of functions in $\Delta$. Then the following conditions hold:
        \begin{enumerate}
            \item The set $I$ is a structure monoid.
            \item For each structure monoid $S$ the associated structures $\Delta_S\subset\Psi_S,\Psi^S\subset \Delta^S$ are structure categories, when defined. 
            \item If $\Delta$ has a single non-identity bijection, then $\Delta = \Delta^I$. \label{non-identity bijection}
            \item Given a structure monoid $S$, then $\Delta_S = \Delta^S$ if and only if $S = \N$.
            \item We have that $\Delta_I\subset\Delta\subset\Delta^I$.

            \item If $0\in I$, then $\Delta$ is one of the structure categories $\Delta_{\{0,1\}},\Delta^{\{0,1\}}$ or $\Delta^\N$
        \end{enumerate}
    \end{lemma}

    \begin{proof}
        \hfill
        \begin{enumerate}
            \item Since $\Delta$ has an identity function on $[1]$, it follows that $1\in I$. Assume that $k,a_1,\ldots, a_n\in I$. Thus there are functions $g,f_i$ for $i\leq k$ where for some fiber of $g$ and $f_i$ attain the size of $k$ and $a_i$, respectively. We may assume that $g$ and $f_i$ have $[1]$ as the codomain, since otherwise we can define such functions using the associated similarity transformation. Now $g\circ (f_1+\ldots +f_k)$ has the fiber of size $a_1+\ldots +a_k$. Thus $a_1+\ldots +a_k\in I$. 
            
            \item Let $S$ be a structure monoid. First, we show that $\Delta^S$ is a structure category. Since identities have fibers of size $1$, it follows that $\Delta^S$ has identities. Assume that $f,g$ are functions in $\Delta$. The fibers of $f+g$ have the same sizes as the fibers of $f$ and $g$. Thus $f+g$ is in $\Delta$. Let $\theta\colon[m]\rightarrow [n]$ be a function in $\Delta$. Let $k_1,\ldots, k_n\in\N$ and denote $L_i = k_{\theta(1)}+\ldots+ k_{\theta(i)}$ and $K_j = k_1+\ldots + k_j$ for $i\leq m$ and $j\leq n$. Consider the component $\theta'_{k_1,\ldots, k_n}\colon [L_m]\rightarrow [K_n]$. Let $j\in [K_n]$. Now $K_{l-1}< j\leq K_l$ for some unique $l\leq n$. By definition $x\xmapsto{\theta'} j$ if and only if $x = L_{i-1} + x'$ where $\theta(i) = l$ and $ x' = j-K_{l-1}$. Thus $\theta'^{-1}(j)$ has the same cardinality as $\theta^{-1}(l)$. Hence $\theta'$ is a function in $\Delta^I$. 

            Next, assume that $0\not\in S$. Thus all the morhisms of $\Delta^S$ are surjections. We show that $\Psi_S$ is a structure category.

            First, $\Psi_S$ is a category, since it has identities and compositions of functions, which is seen as follows:
            Let $[k]\xrightarrow{f} [m]\xrightarrow{g} [n]$ be functions in $\Psi$. Let $i\leq j\leq n$. Now 
            \begin{align*}
            \min((gf)^{-1}\{i\}) 
            &= \min(f^{-1}g^{-1}\{i\}) \\
            &= \min(f^{-1}\{\min(g^{-1}\{i\}\}))\\
            &\leq \min(f^{-1}\{\min(g^{-1}\{j\}\}))\\
            &= \min((gf)^{-1}\{j\})
            \end{align*}
            The coproduct of two functions in $\Psi_S$ is in $\Psi_S$. To see the closure under components of similarities, fix a function $\theta\colon [m]\rightarrow[n]$ is in $\Psi_S$. Let $k_1,\ldots, k_n\in\N$ and denote $L_i = k_{\theta(1)}+\ldots+k_{\theta(i)}$ and $K_j = k_1+\ldots+k_j$ for $i\leq m$ and $j\leq n$. We need to show that $\theta' = \theta'_{k_1,\ldots, k_n}\colon [L_m]\rightarrow [K_n]$ is in $\Psi_S$. Consider $j_1\leq j_2\in [K_n]$. Now $K_{l_1-1}< j_1\leq K_l$ and $K_{l_2-1}<j_2\leq K_t$ for some unique $l_1\leq l_2\leq n$. Now $j_a = K_{l_a-1}+ x_a$ for $x_a = j_a - K_{l_a - 1}$ and $a = 1,2$. Let $i_1$ and $i_2$ be the least elements $\theta$ maps to $l_1$ and $l_2$, respectively. Since $\theta$ is in $\Psi_S$, we know that $i_1\leq i_2$. Now $L_{i_a -1} + x_a$ is the least element $\theta'$ maps to $j_a$ for $a = 1,2$. Notice that $i_1< i_2$ implies that $L_{i_1 - 1} + x_1 \leq L_{i_2 - 1}$ and hence $L_{i_1 - 1} + x_1< L_{i_2- 1} + x_2$. If $i_1 = i_2$, then $x_1\leq x_2$ and thus $L_{i_1 -1}+ x_1\leq L_{i_2 - 1} + x_2$. Hence $\theta'$ is in $\Psi_S$. 

            The case $\Psi^S$ being a structure category is similar. Since the generators of $\Delta_S$ are included in $\Psi_S$ and $\Psi^S$ it follows that $\Delta_S\subset \Psi_S,\Psi^S$. 

            \item Assume that $\Delta$ has a single non-identity bijections $f$. A suitable component of a similarity of $f$ yields a bijection $[2]\rightarrow [2]$ where $1\mapsto 2$. By induction, we attain that $\Delta$ has all bijections. Let $g\colon [m]\rightarrow [n]$ be a function in $\Delta$. Denote by $!_k$ the unique function $[k]\rightarrow [1]$ for $k\in\N$. Notice that up to pre-composition with bijection the functions $!^{g^{-1}\{1\}} + \ldots + !^{g^{-1}\{n\}}$ and $g$ are the same. Since $\Delta$ has all bijections and functions $!^k,k\in I$, it follows that $g$ is in $\Delta$.
            
            \item Let $S$ be a structure monoid. Assume that $S = \N$. It suffices to show that the bijection $[2]\rightarrow [2],1\mapsto 2$ is in $\Delta_\N$. Consider the function $(!^{2'}_{2})\circ (!^0+id_{[2]}+!^0)\colon [2]\rightarrow [4]\rightarrow [2]$. This function is the bijection $[2]\rightarrow [2]$ mapping $1\mapsto 2$ and it is in $\Delta_\N$. Thus $\Delta_\N = \Delta^\N$.  

            Assume then that $S\neq \N$. Let us consider the case that $0\in S$, then $S = \{0,1\}$. Notice that $\Delta^S$ consists of all injections. Notice that $\Delta_S$ contains all strictly increasing functions and strictly increasing functions do define a structure category. Thus $\Delta_S$ consists of strictly increasing functions and $\Delta_S\neq \Delta^S$. If $0\not\in S$, then $\Delta_S\subset \Psi_S\subsetneq \Delta^S$. Thus $\Delta_S\subsetneq\Delta^S$. 
            
        \item Consider $n\in I$. Thus there exists a function $\theta$ in $\Delta$ who has a fiber of size $n$. A suitable component of the similarity $\theta'$ yields exactly the function $[n]\rightarrow [1]$ and this is in $\Delta$. Since $\Delta_I$ is generated by such function, $\Delta_I\subset\Delta\subset\Delta^I$.

        \item Assume that $0\in I$. Thus $\Delta_I$ has all strictly increasing functions. Assume that $\Delta$ has a function $f\colon[m]\rightarrow [n]$ that is not strictly increasing. If $f$ is not an injection, then $m\in I$ for some $m>1$ and thus $I = \N$ and so $\Delta = \Delta^\N$. Assume that $f$ is an injection. It suffices to show that $\Delta$ has a non-identity bijection. Now there exists $i<j\leq m$ such that $f(j)<f(i)$. Since $0\in S$ by inclusions and codiagonals we can construct the function $[2]\rightarrow [m]$ and $[n]\rightarrow [2]$ such that composing them with $f$ yields a map $[2]\rightarrow [2]$, where $1\mapsto 2\mapsto 1$. Thus $\Delta = \Delta^I$ and $I$ can only be either $\{0,1\}$ or $\N$. 
        \end{enumerate}
         
    \end{proof}
    
    \begin{example}
    We will give examples of structure categories $\Delta$.
        \begin{enumerate}
            \item For each structure monoid $I\subset\N$ we have the structure categories $\Delta_I$ and $\Delta^I$.\begin{itemize}
                \item If $I = \{0,1\},\{1\},\N_{>0},\N$ we have $\Delta_I$ consists of strictly increasing maps, identities, functions generated via coproducts and similarities from increasing surjections and all functions, respectively. With the same conditions on $I$, the structure category $\Delta^I$ consists of injections, bijections, surjections and all functions, respectively.
                \item Let $N>0$. Define $I = I_N$ as $\{1,n\mid  n\geq N\}$ or $\{1,N,n\mid n\geq 2N-1\}$. The set $I$ defines a structure monoid and we have four different structure categories $\Delta_I\subset\Psi_I,\Psi^I\subset\Delta^I$. Since each $N>0$ defines a different structure category $\Delta^{I_N}$, it follows that there are infinitely many structure categories.

                \item Let $I$ be a structure monoid where $0\not\in I$. Then $\Psi_I$ consists of those surjections $f\colon [m]\rightarrow [n]$ where $\min f^{-1}\{i\}\leq \min f^{-1}\{j\}$ for $i\leq j\leq n$ and the cardinality of each fiber of $f$ is in $I$. 

                \item Increasing functions do not define a structure category. Consider a non-injective function $\theta\colon [m]\rightarrow [n]$. Assume that the fiber of $f$ at $j\leq n$ has cardinality of at least $2$. Now define $k_l = \begin{cases}
                    0, \text{ if }l\neq j,\\
                    2,\text{ if } l = j.
                \end{cases}$
                Now $\theta (i_1) = j = \theta (i_2)$ for some $i_1<i_2\leq m$. Denote $L_i = k_{\theta(1)}+\ldots + k_{\theta(i)}$ for $i\leq m$. Now $\theta'_{k_1,\ldots, k_n}(L_{i_1}+2) = 2$ but 
                $\theta'_{k_1,\ldots, k_n}(L_{i_2} + 1) = 1$ and $L_{i_1}+2<L_{i_2}+1$. Hence $\theta'_{k_1,\ldots, k_n}$ is not an increasing function. 
            \end{itemize}
            \end{enumerate}
    \end{example}
\subsection{\texorpdfstring{$\Delta$-Multicategories}{Delta-Multicategories}}
    We define $\Delta$–multicategories, where $\Delta$ is a structure category. These are multicategories $C$ equipped with a $\Delta$-action on the hom-sets of $C$. The first two action axioms relating to composition and identities are similar to the ones of monoid axioms, but the latter two require the coproduct and similarity closure properties of a structure category.
    \begin{definition}[$\Delta-$Multicategory]
    Let $\Delta$ be a structure category. A $\Delta$-multicategory $C$ consists of a class of objects $S$ and choice of sets $C(a,b)$ for each $a\in S^*$ and $b\in S$, where $S^*$ is the class $\bigsqcup_{n\in\N} S^n$ of finite words of $S$. For $f\in C(a,b)$, we denote $f\colon a\rightarrow b$ for $a\in S^*$ and $b\in S$.  In addition $C$ is equipped with the following structure:
    \begin{itemize}
        \item A family of maps $\circ$ called the composition: 
        
        $$
        \circ_{a^1,\ldots, a^n,b_1\ldots b_n,c}\colon C(b_1\cdots b_n,c)\times C(a^1,b_1)\times\cdots\times C(a^n,b_n)\rightarrow C(a^1\cdots a^n,b)
        $$
        for $a^1,\ldots, a^n\in S^*$ and $b_1,\ldots, b_n,c\in S$.
        
        \item For each $a\in S$ there is a choice of a morphism $id_a\colon a\rightarrow a$. 
        
        \item For each function $\theta\colon [m]\rightarrow [n]$ in $\Delta$ and $a_1,\ldots, a_n, b\in S$ there is a function
        $$
        \theta_{a_1\cdots a_n,b,*}\colon C(a_{\theta(1)}\cdots a_{\theta(m)},b)\rightarrow C(a_1\cdots a_n, b)
        $$
        \end{itemize}
        \noindent
        Furthermore, for $C$ to be a $\Delta$–multicategory we require the following axioms to be satisfied:
        \begin{itemize}
            \item The associativity of composition:
        {\small
        $$
        (h\circ (g^1,\ldots, g^n))\circ (f^1_1,\ldots, f^1_{m_1},\ldots, f^n_1,\ldots, f^n_{m_n}) = h\circ (g^1\circ (f^1_1,\ldots, f^1_{m_1}),\ldots, g^n\circ (f^n_1,\ldots, f^n_{m_n}))
        $$
        }
        for $h\colon c_1\cdots c_n\rightarrow d, g^i\colon b^i_1\cdots b^i_{m_i}\rightarrow c_i$ and $f^i_j\colon a^{i,j}\rightarrow b^i_j$ where $i\leq n$ and $j\leq m_i$.
        \item The identity laws: 
        $$
        f\circ (id_{a_1},\ldots, id_{a_n}) = f
        $$
        and
        $$
        id_b\circ f = f
        $$
        for $f\colon a\rightarrow b$, where $a = a_1\ldots a_n$ and $a_i,b\in S$ for $i\leq n$. 
        \item The axioms of a $\Delta$-action:
            \begin{enumerate}
                \item $$
                (id_{[n]})_{a,b,*} = id_{C(a,b)}
                $$
                for $n\in\N$, $a\in S^*$ and $b\in S$, where the length of the word $a$ is $n$.
                \item $$
                (\tau\sigma)_{a,b,*} = \tau_{a,b,*}\circ \sigma_{a_{\tau(1)}\cdots a_{\tau(m)},b,*}
                $$
                for functions $[k]\xrightarrow{\sigma}[m]\xrightarrow{\tau}[n]$ in $\Delta$ and $a\in S^*, l(a) = n,$ and $b\in S$.
                
                \item 
                $$
                g\circ (\sigma^1_{a^1,b_1,*}(f_1),\ldots ,\sigma^n_{a^n,b_n,*}(f_n)) = (\sigma^{1}+\cdots + \sigma^{m})_{a^1\cdots a^n,c,*}(g\circ (f_1,\cdots, f_{n}))
                $$
                for functions $\sigma^i\colon [k_i]\rightarrow [l_i],i\leq n$, in $\Delta$ and multimorphisms $g\colon b_{1}\cdots b_{n}\rightarrow c$ and $f_i\colon a^i_{\sigma^i(1)}\cdots a^i_{\sigma^i(k_i)} \rightarrow b_i$ where $a^i = a^i_1\cdots a^i_{k_i}$ for $a^i_1,\ldots,a^i_{k_i}\in S$ and $i\leq n$. 
                
                \item $$
                \tau_{b_1\cdots b_n, c,*}(g)\circ (f_1,\ldots, f_n) = (\tau'_{k_1,\ldots, k_n})_{a^1\cdots a^n,c,*}(g\circ (f_{\tau(1)},\ldots, f_{\tau(m)}))
                $$
                for function $\tau\colon [m]\rightarrow [n]$ and multi-morphisms $f_i\colon a^i\rightarrow b_i, l(a^i) = k_i,$ and $g\colon b_{\tau(1)}\cdots b_{\tau(m)}\rightarrow c$, $i\leq n$.
            \end{enumerate}
    \end{itemize}

    When the category $\Delta$ is the wide subcategory of $\textbf{FinOrd}$ consisting of all functions, bijections or identities, the corresponding $\Delta$–multicategories are called cartesian multicategories, symmetric multicategories and multicategories, respectively.
\end{definition}

\begin{definition}[$2$-category of $\Delta$–multicategories]
    Let $\Delta$ be a structure category and let $M$ and $N$ be $\Delta$–multicategories with classes of objects $R$ and $S$, respectively. A morphism $T\colon M\rightarrow N$ of $\Delta$–multicategories consists of a map $T\colon R\rightarrow S$ and functions $T = T_{a,b}\colon M(a,b)\rightarrow N(\overline{T}(a),T(b))$, where $a\in R^*, b\in R,$ and $\overline{T}\colon R^*\rightarrow S^*$ is the unique monoid morphism extending $T\colon R\rightarrow S$, also denoted as $T$. Additionally, $T$ respects composition, identities and the $\Delta$-action:
    \begin{itemize}
        \item $T(g\circ (f_1,\ldots, f_n)) = T(g)\circ (T(f_1),\ldots, T(f_n))$.
        \item $T(id_c) = id_{T(c)}$
        \item $T(\theta_{a,b,*}(g)) = \theta_{\overline{T}(a),T(b),*}(T(g))$. 
    \end{itemize}
    A natural transformation $\eta\colon T_1\Rightarrow T_2\colon M\rightarrow N$ consists of a family of morphisms $\eta_r\colon T_1(r)\rightarrow T_2(r)$ making the multicategorical diagram
    $$
    \begin{tikzcd}
T_1(a) \arrow[r, "T_1(f)"] \arrow[d, "\eta_a"'] & T_1(b) \arrow[d, "\eta_b"] \\
T_2(a) \arrow[r, "T_2(f)"']                     & T_2(b)                    
\end{tikzcd}
    $$
    commute for each multimorphism $f\colon a\rightarrow b$ in $M$. Here $\eta_a$ denotes the sequence $\eta_a = (\eta_{a_1},\ldots, \eta_{a_n})$ for $a = a_1\cdots a_n$ where $a_1,\ldots, a_n\in R$. The vertical and horizontal composition is defined similarly as in the $2$-category of categories. This determines the $2$-category $\Delta$-$\textbf{MultiCat}$ of $\Delta$–multicategories.
\end{definition}

\begin{lemma}
    Let $\Delta$ be a structure category and let $(T,\mu\colon T^2\Rightarrow T, \lambda\colon I\Rightarrow T) = T\colon M\rightarrow M$ be a monad in the $2$-category $\Delta$-$\textbf{MultiCat}$. Then the underlying Kleisli and Eilenberg-Moore categories associated to $T$ extend into a $\Delta$–multicategories. 
\end{lemma}
\begin{proof}
    We define the Kleisli category $M_T$ as follows:
    \begin{itemize}
        \item The objects of $M_T$ are the same as the objects of $M$. 
        \item A multi-morphism $a\rightarrow b$ in $M_T$ consists of a multimorphism $f\colon a\rightarrow T(b)$ in $M$. 
        \item Let $g\colon b_1\cdots b_n\rightarrow c$ and $f_i\colon a^i\rightarrow b_i$ be multimorphisms for $i\leq n$ in $M_T$. We define the composition as follows:
        $$
        g\circ_{M_T}(f_1,\ldots, f_n) \coloneq \mu_c\circ_M T(g)\circ_M(f_1,\ldots, f_n).
        $$
        The unit of the monad $T$ defines the identity morphisms of $M_T$. 
        \item Let $\theta\colon [m]\rightarrow [n]$ be a function in $\Delta$. We define $\theta^T_{a,b,*}\colon M_T(a_\theta,b)\rightarrow M_T(a,b)$ as the function $\theta_{a,Tb,*}\colon M(a_\theta,T(b))\rightarrow M(a,T(b))$, where $a = a_1\cdots a_n, a_\theta = a_{\theta(1)}\cdots a_{\theta(m)}$ and $a_1,\ldots, a_n, b$ are objects of $M$. 
    \end{itemize}
    We show that $M_T$ is a $\Delta$-multicategory:
    \begin{enumerate}
        \item To show that the identity laws hold, fix $f\colon a = a_1\cdots a_n\rightarrow b$ in $M_T$, where $a_1,\ldots, a_n$ and $b$ are objects of $M$. Now
        $$
        f\circ_{M_T}(\lambda_{a_1},\ldots, \lambda_{a_n}) = \mu_b \circ T(f)\circ (\lambda_{a_1},\ldots, \lambda_{a_n}) = \mu_b\lambda_{Tb} f = f
        $$
        and
        $$
        \lambda_b\circ_{M_T} f = \mu_{b}\circ T(\lambda_b)\circ f = f.
        $$
        Thus the identity laws hold.
        \item For associativity, fix multimorphisms $h\colon c_1\cdots c_n\rightarrow d$, $g_i\colon b^i\rightarrow c_i$ and $f^i_j\colon a^{i,j}\rightarrow b^i_j$ for $b^i = b^i_1\cdots b^i_{m_i}$ for $i\leq n$ and $j\leq m_i$. Now
        \begin{align*}
            &(h\circ_{M_T} (g_1,\ldots, g_n))\circ_{M_T}(f^1_1,\ldots, f^n_{m_n})\\
            &=(\mu_d T(h) \circ (g_1,\ldots, g_n))\circ_{M_T}(f^1_1,\ldots, f^n_{m_n})\\
            &= \mu_d T(\mu_d T(h)\circ (g_1,\ldots, g_n))\circ (f^1_1, \ldots, f^n_{m_n})\\
            &= (\mu_d T(\mu_d) TT(h)\circ (T(g_1),\ldots, T(g_n)))\circ (f^1_1,\ldots, f^n_{m_n})\\
            &= (\mu_d \mu_{Td} TT(h)\circ (T(g_1),\ldots, T(g_n)))\circ (f^1_1,\ldots, f^n_{m_n})\\
            &= ((\mu_d T(h)\circ (\mu_{c_1},\ldots, \mu_{c_n}))\circ (T(g_1),\ldots, T(g_n)))\circ (f^1_1,\ldots, f^n_{m_n})\\
            &= (\mu_d T(h)\circ (\mu_{c_1} T(g_1),\ldots, \mu_{c_n}T(g_n)))\circ (f^1_1,\ldots, f^n_{m_n})\\
            &= \mu_d T(h)\circ (\mu_{c_1} T(g_1)\circ (f^1_1,\ldots, f^1_{m_1}),\ldots, \mu_{c_n} T(g_n)\circ (f^n_1,\ldots, f^n_{m_n}))\\
            &= \mu_d T(h)\circ (g_1\circ_{M_T}(f^1_1,\ldots, f^1_{m_1}),\ldots, g_n\circ_{M_T}(f^n_{1},\ldots, f^n_{m_n}))\\
            &= h\circ_{M_T} (g_1\circ_ {M_T}(f^1_1,\ldots, f^1_{m_1}),\ldots, g_n\circ_{M_T}(f^n_1,\ldots, f^n_{m_n}))
        \end{align*}
        \item The $\Delta$–action axioms are satisfied since $T$ respects the action and thus $M_T$ is a $\Delta$–multicategory.
    \end{enumerate}

    We define the Eilenberg-Moore $\Delta$-multicategory $M^T$ as follows:
    \begin{itemize}
        \item The objects are pairs $(a,f\colon Ta\rightarrow a)$ where the diagrams 
        $$
        \begin{tikzcd}
TTa \arrow[r, "\mu_a"] \arrow[d, "Tf"'] & Ta \arrow[d, "f"] & a \arrow[r, "\lambda_a"] \arrow[rd, "id"'] & Ta \arrow[d, "f"] \\
Ta \arrow[r, "f"']                      & a                 &                                            & a                
\end{tikzcd}
        $$
        commute.
        \item A multimorphism $(a_1,f_1)\cdots (a_n,f_n)\rightarrow (a,f)$ in $M^T$ consist of a multimorphism $\alpha\colon a_1\cdots a_n\rightarrow a$ in $M$ making the diagram
        $$
        \begin{tikzcd}
Ta_1\cdots Ta_n \arrow[r, "T(\alpha)"] \arrow[d, "{(f_1,\ldots, f_n)}"'] & Ta \arrow[d, "f"] \\
a_1\cdots a_n \arrow[r, "\alpha"']                                       & a                
\end{tikzcd}
        $$
        commute.
        \item The composition in $M^T$ is defined as in $M$.
        \item The action $\theta_{a,b,*}\colon M^T(a_{\theta},b)\rightarrow M^T(a,b)$ is defined as in $M$ and a quick check shows that it is well defined.
    \end{itemize}
    Associativity and unitality axioms of a multicategory are satisfied making $M^T$ a $\Delta$–multicategory.
\end{proof}

\begin{example}Here we give examples of different $\Delta$-multicategories.
    \begin{enumerate} 
        \item Each (cartesian/symmetric) monoidal category determines a (cartesian/symmetric) multicategory. 
        \item If $C$ is a $\Delta$-multicategory, then any full submulticategory of $C$ is also a $\Delta$–multicategory. Denote the class of objects of $C$ by $S$. Let $P$ be any subclass of $S$. We define the full submulticategory $A$ over $P$ as the one which has $P$ as the class of objects and $A(a,b) = C(a,b)$ for $a\in P^*$ and $b\in P$. The $\Delta$-action restricts into $A$ and thus $A$ has the structure of a $\Delta$-multicategory.
        \item The category of compactly generated Hausdorff spaces has two different cartesian multicategory structures, one defined by its categorical product and the other defined by restricting the cartesian structure of topological spaces.
        \item The cartesian multicategory of sets is specifically defined as follows:
        \begin{itemize}
            \item The objects are sets $X$.
            \item A morphism $f\colon X_1\cdots X_n\rightarrow X$ is just a function $\Pi_{i\leq n} X_i\coloneq \{(x_1,\ldots, x_n)\mid  x_i\in X_i,i\leq n\}\rightarrow X$. 
            \item Composition is defined using usual function composition.
            \item Let $\theta\colon [m]\rightarrow [n]$ be a function between finite ordinals and let $f\colon X_{\theta(1)}\cdots X_{\theta(m)}\rightarrow X$ be a multimorphism. We define $\theta_*(f)(x_1,\ldots, x_n) = f(x_{\theta(1)},\ldots, x_{\theta(m)})$ for $x_i\in X_i$ for $i\leq n$. 
        \end{itemize}
        We will show only the last condition of the action axiom:
        Let $\theta\colon[m]\rightarrow [n]$ and let $g\colon Y_{\theta(1)}\cdots Y_{\theta(m)}\rightarrow Y$ and $f_i\colon X^1\rightarrow Y_i$ for $i\leq n$. We show that 
        $$
        \theta_*(g)\circ (f_1,\ldots, f_n) = (\theta'_{k_1,\ldots, k_n})_*(g\circ (f_{\theta(1)},\ldots, f_{\theta(m)})),
        $$
        where $k_i$ is the length of the word $X^i$ for $i\leq n$. Let $x_i\in \Pi X^i$ for $i\leq n$. Now
        \begin{align*}
            \theta_*(g)\circ (f_1,\ldots, f_n)(x_1\cdots x_n)
            &= \theta_*(g)(f_1(x_1),\ldots, f_n(x_n))\\
            &= g(f_{\theta(1)}(x_{\theta(1)}),\ldots, f_{\theta(m)}(x_{\theta(m)}))\\
            &= (\theta_{k_1,\ldots, k_n}')_*(g\circ (f_{\theta(1)},\ldots,f_{\theta(m)}))(x_1,\ldots, x_n).
        \end{align*}
        Thus $\theta_*(g)\circ (f_1,\ldots, f_n) = \theta'_*(g\circ(f_1,\ldots, f_n))$.

        \item Interestingly, the category of relations on sets has a cartesian monoidal structure given by its cartesian monoidal structure defined by the disjoint union of sets and hence it has a cartesian multicategory structure. There is another surjective-multicategory structure on the category of relations determined by the powerset monad $(\mathcal{P},\mu\colon \mathcal{PP}\Rightarrow \mathcal{P}, \lambda\colon I\Rightarrow \mathcal{P})$ on the surjective-multicategory of sets. We set
        $$
        \mathcal{P}(f\colon X_1\cdots X_n\rightarrow X)(A_1,\ldots, A_n) = f(A_1\times\cdots\times A_n)
        $$
        and
        $$
        x\xmapsto{\lambda_X} \{x\}\colon X\rightarrow\mathcal{P}(X)\text{ and }\mathcal{A}\xmapsto{\mu_X}\bigcup\mathcal{A}\colon \mathcal{PP}(X)\rightarrow\mathcal{P}(X)
        $$
        for sets $X, X_i$ and $A_i\subset X_i,i\leq n$. It is fast to check that $\mathcal{P}$ is a morphism between the multicategories. It is equivariant with respect the action of $\Delta^{\mathbb{N}_1}$ making $\mathcal{P}$ into a monad on the surjective multicategory of sets. Thus, as the Kleisli-category of the powerset monad, the category of relations inherits a surjective multicategory structure. Let $\theta\colon [1]\hookrightarrow[2]$ and $f\colon [1]\rightarrow [1]$. Consider $\theta_{[1]\emptyset, [1],*}(f)\colon [1]\emptyset \rightarrow [1]$ which is the empty map. Notice that $\mathcal{P}(\theta_*(f))$ is not a surjection, but $\theta_*(\mathcal{P}(f))$ is. Hence $\mathcal{P}$ is not an endomorphism on the cartesian multicategory of sets.
        
        \item The category of sets equipped with partial functions is a surjective multi-category as the Kleisli multicategory of the list-monad $T$. The failure of $\textbf{Set}^T$ to be a cartesian multicategory is related to the failure of $\textbf{Set}^\mathcal{P}$ to be a cartesian multicategory.
        \end{enumerate}
\end{example}
\section{Universal algebra}
We define the syntax and semantics of universal algebra. Syntax contains notions of a signature $\sigma$, a context structure $R$, an $R$–equational theory $E$ and an $R$-deduction system $\vdash_R$. Semantics concentrates on the models inhabiting in different multicategories and we define the connection between syntax and semantics via the satisfiability of models. Once the concepts are bolted down, we prove the soundness and completeness of some of the associated deduction systems.
\subsection{Syntax}
Syntax refers to the construction of mathematical structures out of symbol manipulation itself. We define a symbol as an element of a set and using recursive constructions on symbols we create the languages appropriate to our setting.
\begin{definition}[Signature of universal algebra]
    A signature $\sigma$ consists of the following data:
    \begin{itemize}
        \item A set $S$ of sorts.
        \item A set of function symbols $M$ with a typing $t\colon M\rightarrow S^*\times S$. We denote $f\colon a\rightarrow b$ for $f\in M$ where $f\xmapsto{t} (a,b)$ and if $a = ()$ is the empty sequence, then we say that $f$ is a constant symbol and $f\colon b$. 
        \item A set of typed variables $t'\colon V\rightarrow S$. Notation $x\colon s$ means that $x\xmapsto{t'} s$. We require the fibers of $t'$ to be countably infinite.
    \end{itemize}
    We will often shorten and denote $\sigma = (S,M)$ or $\sigma = (S,M,V)$. 
\end{definition}
\begin{definition}[Terms]
    Fix a signature $\sigma = (S,M,V)$. We define a typed set Term $\ra S$ of terms as follows:
    \begin{itemize}
        \item $c\colon s,x\colon s\in \text{Term}$ for $c\colon s\in M$ and $x\colon s\in V$. 
        \item We define $f(t_0,\ldots, t_n)\colon b\in \text{Term}$ for $f\colon a_0\cdots a_n\rightarrow b\in M$ and $t_i\colon a_i\in\text{Term}$ for $i\leq n$. 
    \end{itemize}
    We define a function $\tau\colon \text{Term}\ra V^*$, where $\tau(t) = \begin{cases}
        t, \text{ if $t$ is a variable symbol}\\
        (),\text{ if $t$ is a constant symbol}\\
        \tau(t_1)\ldots \tau(t_n),\text{ if $t = f(t_1,\ldots,t_n)$}.
    \end{cases}$
    Let $t$ be a term and denote $\tau(t) = x_1\ldots x_n$. The set of variables of a term $t$ is $\text{Var}(t) = \{x_1,\ldots, x_n\}$ and similarly for words $v\in V^*$.
\end{definition}

\begin{definition}[Context Structure]
    Let $V$ be a set, whose elements we call letters. We denote the set of finite words of $V$ as $V^*$. We say that a word $v\in V^*$ is a context of $V$ if no element in $v$ repeats and we denote the set of the contexts of $V$ with $C$. We call a relation $R\subset C\times V^*$ a context structure if it satisfies the four conditions:
    \begin{enumerate}
        \item $cRc$ for every $c\in C$.
        \item $cRv$ implies that all letters in $v$ are expressed in $c$.
        \item $cRv^1\cdots v^n$ and $v^iRw^i,$ for $v^i,w^i\in V^*$ and $i\leq n$, implies $cRw^1\cdots w^n$. 
        \item If $cRv$, $s\colon \text{Var}(c)\rightarrow V^*$ and $dRs(c)$, then $dRs(v)$. Here $s(v) = s(v_1)\cdots s(v_n)$ where $v = v_1\cdots v_n$ and $v_1,\ldots, v_n\in V$. 
    \end{enumerate}
     We call a word $v\in V^*$ an $R$-word if $cRv$ for some $c\in C$ and in such case we say that $c$ is an $R$-context for $v$.
     \begin{itemize}
         \item A context structure $R\subset C\times V^*$ on $V$ is said to be modelable if $cRv^1\cdots v^n$ implies that there exists contexts $c^1,\ldots,c^n$ where $cRc^1\cdots c^n$ and $c^iRv^i$ for $i\leq n$. 
        \item A context structure $R$ is said to be balanced if $vRw$ implies that $\text{Var}(v) = \text{Var}(w)$.
        \item A balanced and modelable context structure is called balanced-complete.
     \end{itemize}
     
\end{definition}

There is a correspondence between structure categories $\Delta$ and context structures $R$:
\begin{theorem}
    Let $V$ be an infinite set. Then the construction $R\mapsto \Delta_R$ defines a bijection from context structures to structure categories, where $\Delta_R$ consists of functions $\theta\colon [m]\rightarrow [n]$ where $c_1\ldots c_n R c_{\theta(1)}\cdots c_{\theta(m)}$ and $c = c_1\cdots c_n$ is a context where $c_1,\ldots, c_n\in V$. 
\end{theorem}
\begin{proof}
    Fix a context structure $R$ on $V$ and fix an injection $c\colon\N\rightarrow V$. Notice that $c^n\coloneqq c_1\cdots c_n$ is a context for each $n\in\N$. We define a structure category $\Delta_R$ to be the wide subcategory of $\textbf{FinOrd}$ having the functions $f\colon [m]\rightarrow [n]$, where $c^nRc_f$, where $c_f\coloneqq c_{f(1)}\cdots c_{f(m)}$. We will often use $d_\theta$ to denote $d_{\theta(1)}\cdots d_{\theta(m)}$ where $d = d_1\cdots d_n, d_i\in V, i\leq n$ and $\theta\colon [m]\rightarrow [n]$.

   We will show that $\Delta_R$ is a structure category. Since $c^nRc^n$, it follows that $id\colon [n]\rightarrow [n]$ is a morphism in $\Delta_R$ for each $n\in\N$. Assume that $[k]\xrightarrow{f}[m] \xrightarrow{g} [n]$. We need to show that $g\circ f\colon [k]\rightarrow [n]$ is a morphism in $\Delta_R$. Now we have that $c^mRc_f$ and $c^nRc_g$. Consider the renaming $s\colon c_i\mapsto c_{g(i)}$ for $i\leq m$. Notice that $s(c^m) = s(c_1)\cdots s(c_m) =  c_g$. Since $c^nRs(c^m) = c_g$ and $c^mR c_f$, it follows that $c^nRs(c_f) = c_{g\circ f}$.

   Now we show that $\Delta_R$ is closed under the coproduct of morphisms. Let $f\colon[m_i]\rightarrow [n_i]$ be a morhpism in $\Delta_R$ for $i = 1,2$. We show that $c^{n_1+n_2}Rc_{f_1(1)}\cdots c_{n_1+f_2(m_2)}$. By definition of $f_1$ and $f_2$ being function in $\Delta_R$, we have that $c^{n_1}Rc_{f_1}$ and $c^{n_2}Rc_{f_2}$. Consider $c_i\xmapsto{s} c_{n_1+i}$ for $i\leq n_2$. Now $s(c^{n_2}) = c_{n_1+1}\cdots c_{n_1+n_2}Rc_{n_1+f_2(1)}\cdots c_{n_1+f_2(m_2)} = s(c_{f_2})$. Since $c^{n_1}Rc_{f_1}$ and $c_{n_1+1}\cdots c_{n_1+n_2} R c_{n_1+f_2(1)}\cdots c_{n_1+f_2(m_2)}$, we attain that $c^{n_1+n_2}Rc_{f_1(1)}\cdots c_{n_1+f_2(m_2)} = c_{f_1+f_2}.$

   Next we show that $\Delta_R$ satisfies the similarity condition. Let $\theta\colon [m]\rightarrow [n]$ be in $\Delta$. Let $k_1,\ldots, k_n\in\N$. We denote $L_i = k_{\theta(1)}+\ldots + k_{\theta(i)}$ and $K_j = k_1+\cdots + k_j$ for $i\leq m$ and $j\leq n$. We need to show that the map $\theta'_{k_1,\ldots, k_n} = \theta' \colon [L_m] \rightarrow [K_n]$ is a morphism in $\Delta$. By $c^{i,j}$ we mean the context $c_{i+1}\cdots c_{j}$ for $i\leq j$. We show that $c^{K_n} R c_{\theta'}$. Let $d_i = c^{K_{i-1},K_i}$ for $i\leq n$. Consider the mapping $c_i\xmapsto{s}d_i$ for $i\leq n$. Now $c^{K_n} = d_1\cdots d_n$ and so $c^{K_n} R d_{\theta(1)}\cdots d_{\theta(m)}$. Notice that $d_{\theta(1)}\cdots d_{\theta(m)} = c_{\theta'}$. This is seen by the fact that the the variable at position $L_{i-1}+j$ in $d_{\theta(1)}\cdots d_{\theta(m)}$ is $c_{K_{\theta(i)-1} + j} = c_{\theta'(L_{i-1}+j)}$ for $0<j\leq k_{\theta(i)}$. 
   
   Conversely given a structure category $\Delta$,  we set $(c,v)\in R_\Delta$, if and only if $v = c_\theta$ for some $\theta\colon [l(v)]\rightarrow [l(c)]$. Now we show that $R_\Delta$ is a context structure. The first two conditions follow directly:  Clearly $cRc$ for all contexts $c\in V^*$ and if $cRv$, then $v = c_f = c_{f(1)}\cdots c_{f(l(v))}$ for some function $f$ and thus all variables in $v$ are expressed in $c$.
   
  Assume that $cRv_1\cdots v_n$ and $v_iRw_i$ for $i\leq n$ with $v_i, w_i$ and $c$ having lengths $k_i,l_i,j\in\N$, respectively for $i\leq m$. Thus we have functions $g\colon [k_1+\ldots + k_n]\rightarrow [j]$ and $f_i\colon [l_i]\rightarrow [k_i]$ for $i\leq n$ where $v_1\cdots v_n = c_{g(1)}\cdots c_{g(k_1+\ldots +k_n)}$ and $w_i = v_{i,f_i(1)}\cdots v_{i,f_{i}(l_i)}$. Thus we have $g\circ (f_1+\ldots + f_n)\colon [l_1+\ldots + l_n]\rightarrow [j]$ in $\Delta$ and now $c_{g\circ(f_1+\ldots + f_n)} = w_1\cdots w_n$.
    
  Lastly, let $cRv$, $s\colon \text{Var}(c)\rightarrow V^*$ and assume that $dRs(c)$. Thus $v = c_\theta$ and $s(v) = d_\phi$ for some $\theta\colon[m]\rightarrow [n]$ and $\phi\colon[k_1+\ldots+ k_n]\rightarrow [l]$ where $m,n$ and $k_i$ are the lengths of $v$, $c$ and $s(v_i)$, respectively for $i\leq n$. We are to show that $dRs(v)$. Notice that $s(v) = d_{\phi\circ \theta'_{k_1,\ldots, k_n}}$ which shows that $dRs(v)$.

   These mappings $R\mapsto \Delta_R$ and $\Delta\mapsto R_\Delta$ are inverses of each other
   which proves the claim.
\end{proof}

Since we have bijective correspondence between structure categories and context structures on an infinite set, we are able to many examples of context structures.
\begin{example} Let $V$ be an infinite set and let $x,y,z\in V$ be pairwise different elements of $V$. Assume that $c$ is a context of $V$ and $v$ a word of $V$. Consider the following context structures on $V$:
\begin{enumerate}
    \item Let $I$ be a structure monoid.
    \begin{itemize}
        \item There is a context structure $R^I$, where $cRv$ if and only if each variable $x$ in $v$ is in $c$ and expressed $n_x$ many times where $n_x\in I$. If $0,3\in I$, then $xyzR^I zzzy$.
            \begin{itemize}
                \item We call $R^{\{0,1\}}$ the injective context structure and notice that $cR^{\{0,1\}} v$ if and only if each variable in $v$ is expressed at most once in the context $c$: $xyz R^{\{0,1\}} xz$.
                \item The context structure $R^\N$ is called the cartesian context structure and $cR^\N v$ if and only if the variables in $v$ are expressed in context $c$: $xyzR^\N yxyxx$.
                \item The surjective context structure is defined to be $R^{\{m\geq 1\}}$ and notice $cR^{\{m\geq 1\}} v$ if and only if all the variables in $v$ are expressed in $c$: $xyzR^{\{m\geq 1\}} xyyzx$
                \item The context structure $R^{\{1\}}$ is called the bijective context structure and $cRv$ if and only if $c$ and $v$ are permutations of each other: $xyzR^{\{1\}}yzx$. 
            \end{itemize}
        \item We have a context structure $R_I$ which is generated by the condition $xRx^n$ for $n\in I$. Each context structure $R$ satisfies $R_I\subset R\subset R^I$ where $I$ is the set consisting $n\in\N$ where $xRx^n$ for $x\in V$. 
            \begin{itemize}
                \item The context structures $R_I$ for $I = \{0,1\}$ and $\{1\}$ are called the strictly increasing and trivial context structures. 
            \end{itemize}
        \item Assume $0\not\in I$. The left surjections of $\Delta^I$ correspond to the context structure $R$ called left surjective $I$-context structure, where $cRv$ if $cR^I v$ and the first appearance of a variable $x$ in $v$ is before the appearance of the variable $y$ in $v$, if the appearance of $x$ in $c$ is earlier than the appearance of $y$ in $c$. Consider as an example $xyR xxyyyx$ if $3\in I$ for different $x,y\in V$. Similarly the right surjections of $\Delta^I$ define the right surjective $I$-context structure (the last appearance of $x$ is earlier than the last appearance of $y$ in $v$ if $x$ appears before $y$ in $c$).
        \item The following eight context structures are modelable and later we will see that there are no others.
            \begin{itemize}
                \item Five are attained as follows: $R^I,R_I$ for $I = \{0,1\}, \{1\},\N$, which corresponds to the injective, strictly increasing, bijective, trivial and cartesian context structures.
                \item We attain three more from the surjective context structure $R^{\{m\geq 1\}}$ and its substructures defined by left and right surjective context structures. 
            \end{itemize}
    \end{itemize}
\end{enumerate}
\end{example}
\begin{definition}
    Let $V$ be an infinite set with a context structure $R$. Let $v\in V^*$ be a word. We call a context $c$ a terminal $R$-context of $v$ if $wRc$ is equivalent with $wRv$ for all contexts $w\in V^*$. We will often denote $c = \overline{v}$. 
\end{definition}
A context structure $R$ has two different terminal contexts for a word if and only if $R = R^I$ for some structure monoid $I$. This is seen from the fact that if $v$ has terminal contexts $c$ and $d$, then they are contexts for each other and the corresponding structure category $\Delta_R$ contains a non-identity bijection and hence $\Delta_R = \Delta^I$ for some structure monoid $I$ by Lemma \ref{structure category lemma}(3).

\begin{theorem}\label{characterizing modelable context-structures}
    Let $V$ be an infinite set. Then there are exactly eight different modelable context structures $R$ on $V$. Especially, all modelable context structures $R$ have terminal contexts for all $R$-word.
\end{theorem}
\begin{proof}
    We have already seen eight different modelable context structures and each has a terminal context for each word. We show that there are no others. Assume that $R$ is a modelable context structure. Let $I$ be the set of multiplicities of variables in $R$-words. If $0\in I$, then we know by Lemma \ref{structure category lemma}(6), that $R$ is either the injective, strictly increasing or cartesian context structure, all of which are modelable context structures. Assume then that $0\not\in I$. If $I = \{1\}$, then $R$ is either the trivial or bijective context structure by Lemma \ref{structure category lemma}(3). Thus we may assume that $I\neq \{1\}$. This implies that $n\in I$ for some $n>1$. Hence $xRxx^{n-1}$, which implies that $xRxx$ by modelability of $R$ for some $x\in V$ and hence $2\in I$. Therefore $I = \N_{1}$. Notice that the surjective context structure $R^I$ is modelable; hence, we may assume that $R\neq R^I$. 

    We will show that $R$ is either the left or right surjective context structure. Let $x,y\in V$ be different elements. Notice that $xRxx$ and hence $xyRxyxy$ and thus $xyx$ is an $R$-word by the modelability of $R$. Thus $xyRxyx$ or $yxRxyx$. We will show that the first case implies that $R$ is the left surjective context structure and the latter case implies that $R$ is the right surjective context structure. Since the latter case is similar, we may assume that $xyRxyx$. Consider the set $S\subset \N$ consisting of $ n\in\N$, where each word $v$ of length at most $n$ holds that $wRv$ if and only if $\overline{v} = w$, where $\overline{v}$ is the left surjective terminal context of $v$. Notice that $\overline{v}$ contains the variables of $v$ with the latter repetitions of variables deleted. Since $R$ is balanced and not the surjective context structure, it holds that each word has at most a single $R$-context.  Hence it suffices to show that $S = \N$. Clearly $0\in S$. Assume that $n\in S$. We'll show that $n+1\in S$.
    
   Assume that $wRv$. We will show that $w = \overline{v}$. Now $v = av'$ for some $a\in V$ and $v'\in V^*$. If $v'$ is not a context, then $wRac$ for some context $c$ and $cRv'$ and thus $w = \overline{ac} = \overline{av} = \overline{v}$. Hence we may assume that $v'$ is a context of the form $w'aw''$ for some $w',w''\in V^*$. Now $wRaw'aw''$ and thus $wRcw''$ for $cRaw'a$. Since $xyRxyx$, it follows that $aw'Raw'a$ and thus $c = aw'$. Therefore $wRaw'w''$ and thus $w = aw'w'' = \overline{aw'aw''} = \overline{v}$.
    
    Next we show that $\overline{v}Rv$. Let $x$ be the last variable expressed in $\overline{v}$. Let $i$ be the index $x$ is expressed in $v$ for the first time. Now $v = v_1xv_2$ where the length of $v_1$ is $i-1$. Therefore $x$ is not expressed in $v_1$. If $v_1$ is the empty word then $v = x^{n+1}$, which is an $R$-word. Assume that $v_1$ is not empty. Now $v_1$ and $xv_2$ are $R$-words by induction hypothesis. If either $v_1$ or $xv_2$ is not a context, then $\overline{v} = \overline{\overline{v_1}\text{ }\overline{xv_2}}$ is an $R$-context for $v$. Thus we may assume words $v_1$ and $xv_2$ are contexts. We attain that $\overline{v} = v_1x$. Let $y$ be the last variable expressed in $v_1$ and consider the function $s\colon\text{Var}(v_1)\rightarrow V^*$, where $z\mapsto \begin{cases}
        z,\text{ if } z\not = y\\
        yx,\text{ if } z = y
    \end{cases}$.
    By the induction hypothesis, $v_1Rv_1v_2$. Thus $v_1x = s(v_1)Rs(v_1v_2) = v_1xs(v_2)$ and thus $v_1x R v_1\overline{xs(v_2)} = v_1xv_2 = v$. Therefore $\overline{v}Rv$.
\end{proof}
\begin{corollary}\label{Sets of contexts corollary}
        Let $R$ be a modelable context structure on an infinite set $V$. Denote by $C_v$ the set of contexts of a word $v\in V^*$. Then $C_{v^i}\subset C_{w^i}$ for $i\leq n$ implies that $C_{v^1\cdots v^n}\subset C_{w^1\cdots w^n}$ for words $v_i,w_i\in V^*,i\leq n$. Furthermore, $C_{v^1\cdots v^n} = C_{v^{\theta(1)}\cdots v^{\theta(m)}}$ for any surjection $\theta\colon[m]\rightarrow [n]$ in $\Delta_R$ and $v^i\in V^*, i\leq n$.
\end{corollary}
\begin{proof}
    Assume that $C_{v^i}\subset C_{w^i}$ for $v^i,w^i\in V^*$ and $i\leq n$. We show that $C_{v^1\cdots v^n}\subset C_{w^1\ldots w^n}$. Let $wRv^1\ldots v^n$. Since $R$ is modelable, we have that $wR\overline{v^1}\cdots \overline{v^n}$ where $\overline{v^i}$ is an $R$-context for $v^i$ for $i\leq n$. Since $\overline{v^i}\in C_{v^i}$ we have $\overline{v^i}\in C_{w^i}$ and so $\overline{v^i}Rw^i$. Since $wR\overline{v^1}\cdots\overline{v^n}$ and $\overline{v^i}Rw^i$, it follows that $wRw^1\ldots w^n$ and hence $w\in C_{w^1\cdots w^n}$. 

    Let $\theta\colon [m]\rightarrow [n]$ be a surjection in $\Delta_R$. Let $v^1,\ldots, v^n\in V^*$. We show that $C_{v^1\cdots v^n} = C_{v^{\theta(1)}\cdots v^{\theta(m)}}$. If $R$ is the strictly increasing, injective, trivial or bijective context structure, then the equality is clear, since $\theta$ is either an identity or a bijection. Assume then that $R$ is a balanced context structure. Thus each word of $V$ is an $R$-word by Theorem $\ref{characterizing modelable context-structures}$. Notice that 
    $$
    \overline{v^1\cdots v^n} = \overline{v^{\theta(1)}\cdots v^{\theta(m)}}
    $$
    and hence $C_{v^1\cdots v^n} = C_{\overline{v^1\cdots v^n}} = C_{v^{\theta(1)}\cdots v^{\theta(m)}}$. 
\end{proof}

\begin{definition}[R–Theory]
    Let $\sigma = (S,M,V)$ be a signature with $V$ having a context structure $R$. We say that a context $c\in V^*$ is an $R$-context for a term $t$, if $sR\tau(t)$. If $t_1$ and $t_2$ are $\sigma$-terms of type $s$ sharing a context $v$, then $t_1\approx_v t_2$ is called an $R$-equation. A set of $R$-equations is called an $R$-theory. 
\end{definition}

\begin{definition}[Renaming \& Substitution]
    Let $\sigma = (S,M,V)$ be a signature with a context structure $R$. Let $v = x_1\ldots x_n$ be a context. We say that a function $s\colon Var(v) = \{x_1,\ldots, x_n\}\ra$ Term is a renaming of variables in $v$, if $s$ preserves the types, i.e. $s(x)$ has the same type as $x$ for $x\in Var(v)$. Let $s$ be a renaming of variables in $v$. We define $\overline{s}\colon Term_v\ra Term$, where $$
    \overline{s}(t) = \begin{cases}
        s(t),\text{ if $t\in Var(t)$}\\
        t,\text{ if $t$ is a constant symbol}\\
        f(\overline{s}(t_1),\ldots, \overline{s}(t_n)),\text{ if $t = f(t_1,\ldots, t_n)$}
    \end{cases},
    $$
    where the set Term$_v$ consists of those terms $t$, where $Var(t)\subset Var(v).$ We will denote $\overline{s} = s$.
    
    Let $v = v_1\cdots v_n$ be a context. A triple $(s,(w_i)_{i\leq n},w)$ is called an $R$-renaming if $s$ is a renaming of variables in $v$ and the contexts $w$ and $w_i$ are $R$-contexts for $w_1\cdots w_n$ and $s(v_i)$, respectively, for $i\leq n$. 
\end{definition}
\begin{lemma}
    Let $\sigma = (S,M,V)$ be a signature with a context structure $R$. Assume that $t$ has an $R$-context $v = v_1\cdots v_n$, where $v_i\in V$ for $i\leq n$. Assume that $(s,w,(w_i)_i)$ is an $R$-renaming of variables in $v$. Then $w$ is an $R$-context for $s(t)$.
\end{lemma}
\begin{proof}
    By induction we see that $\tau(s(t)) = \tau(s(v_{\theta(1)}))\cdots\tau(s(v_{\theta(m)}))$ for the $\theta\colon[m]\rightarrow [n]$ where $\tau(t) = v_{\theta(1)}\cdots v_{\theta(m)}$. Since we have $wRw_1\cdots w_n$ and $w_iR\tau(s(v_i))$ for $i\leq n$, it follows by the third and fourth conditions of context structures, respectively, that $w R \tau(s(v_1))\cdots \tau(s(v_n))$ and $wR\tau(s(v_{\theta(1)}))\cdots \tau(s(v_{\theta(m)})) = \tau(s(t))$.
\end{proof}

\begin{definition}[Syntactic deduction]
    Let $\sigma$ be a signature with a context structure $R$. Let $E$ be an $R$-theory. We define the set $D_E^R = D_E = D$ of $R$-deduced equations from $E$ as the smallest set of $R$-equations satisfying the conditions:
    \begin{itemize}
        \item $E\subset D$.
        
        \item $t\approx_v t\in D$ for any term $t$ in $R$-context $v$.
        
        \item If $t\approx_v t'\in D$, then $t'\approx_v t\in D_E$.
        
        \item If $t_1\approx_v t_2, t_2\approx_v t_3\in D$, then $t_1\approx_v t_3\in D$.
        \item Let $t_1\approx_v t_2\in D$, where $v = v_1\cdots v_n$. Let $(s_1,w,(w_i)_{i\leq n})$ and $(s_2,w,(w_i)_{i\leq n})$ be $R$-renamings of $v$ such that $s_1(v_i)\approx_{w_i} s_2(v_i)\in D_E$ for $i\leq n$. Then $s_1(t_1)\approx_w s_2(t_2)\in D$. 
    \end{itemize}
    We denote $E\vdash_R t_1\approx_v t_2$ if and only if $t_1\approx_v t_2\in D$. We say that a context structure $R$ is $\sigma$-deductively complete if for all $R,\sigma$-theories $E\cup\{\phi\}$ it holds that $E\vdash_R \phi$ if and only if $E\vdash_\text{Cart}\phi$, where Cart is the maximal context structure.
\end{definition}

\subsection{Soundness}
For a given signature $\sigma$ with a modelable context structure $R$, we define models of $R$-theories in any given $\Delta_R$-multicategory. We show that the deductions system $\vdash_R$ is sound and complete with respect to all models in $\Delta_R$-multicategories. By soundness and completeness, we mean the inclusions $D^R_E\subset\bigcap_{m}\text{True}(m)$ and $\bigcap_m\text{True}(m)\subset D_E^R$, respectively, where the indexing is over all $\sigma$-model $m$ in a $\Delta_R$-multicategory $C$. Here $D^R_E$ means all the equations deduced from $E$ and $\text{True}(m)$ is the set of all equations true in $m$. Later we see that completeness is attained for all eight different cases for modelable $R$, but $\textbf{Set}$-completeness is attained for balanced-complete $R$ and the cartesian context structures $R$ (six cases).
\begin{definition}[Model]
    Let $\sigma = (S,M,V)$ be a signature. Let $C$ be a multi-category. A $\sigma$-model in $C$ is an association $m$, where
    \begin{itemize}
        \item $m(s)$ is an object of $C$ for each sort $s\in S$. For $a = a_1\cdots a_n\in S^*$, where $a_i\in S$ for $i\leq n$, we define $m(a) = m(a_1)\cdots m(a_n).$
        \item $m(f)\colon m(a)\ra m(b)$ is a multi-morphism in $C$ for each morphism symbol $f\colon a\rightarrow b\in M$. 
    \end{itemize}
    A morphism of models $m\ra n$ consists of a family $f$ of functions $f(s)\colon m(s)\ra n(s)$ for each sort $s$ and where the diagram
    $$
    \begin{tikzcd}
m(a) \arrow[r, "m(\alpha)"] \arrow[d, "f(a)"'] & m(b) \arrow[d, "f(b)"] \\
n(a) \arrow[r, "n(\alpha)"']                   & n(b)                  
\end{tikzcd}
    $$
    commutes for each morphism symbol $\alpha\colon a\ra b$. Here we understand the commutativity of the diagram to mean that $f(b)\circ m(\alpha) = n(\alpha)\circ f(a)$, where $f(a) = (f(a_1),\ldots, f(a_n))$ for $a = a_1\cdots a_n$ and $a_i\in S,i\leq n$.
\end{definition}

\begin{definition}[Canonical morphisms]
    Let $\sigma = (S,M,V)$ be a signature with a context structure $R$. Let $m$ be a $\sigma$-model in a $\Delta_R$–multicategory $C$. Let $v_1,\ldots v_n\in V$ and $w = v_{\theta(1)}\cdots v_{\theta(m)}$ for a function $\theta\colon[m]\rightarrow [n]$ in $\Delta_R$.  We set $m_v = m_{v_1}\cdots m_{v_n}$ and $m_{v_i} = m(a_i)$ for $v_i\colon a_i$ and $i\leq n$. We define a function $m_{v,w,b}\colon C(m_w,b)\rightarrow C(m_v,b)$ as $\theta^*_{m_v,b}$. Instead of writing $m_{v,w,b}(f)$ for a multimorphism $f\colon m_w\rightarrow b$ we use the notation $f*m_{v,w}\colon m_v\rightarrow b$ for the sake of convenience. Even though the function $m_{v,w}$ is not a morphism in $C$, it behaves very similarly to a canonical morphism constructed from the natural transformations in a monoidal category and so we choose to call $m_{v,w}$ a canonical morphism.  
\end{definition}
\begin{lemma}\label{Properities of canonical morphisms}
    Let $m$ be a $\sigma$-model in a $\Delta_R$–multicategory $C$ where $R$ is a context structure on the variables of $\sigma$. Then the following assertions hold:
    \begin{enumerate}
        \item Let $v$ be a context. Then $f*m_{v,v} = f$ for any $f\colon m_v\rightarrow b$ in $C$.
        
        \item Let $cRv$ and $vRw$. Then $(f*m_{v,w})*m_{c,v} = f*m_{c,w}$ for any $f\colon m_w\rightarrow b$. 
        
        \item Let $v_iRw_i$ for $i\leq n$ and assume that $cRv_1\cdots v_n$. Let $f_i\colon m_{w_i}\rightarrow b_i$ for $i\leq n$ and let $g\colon b_1\cdots b_n\rightarrow c$ be multimorphisms in $C$. Then
        $$
        g\circ (f_1*m_{v_1,w_1},\ldots, f_n*m_{v_n,w_n})*m_{c,v_1\cdots v_n} = (g\circ (f_1,\ldots, f_n))*m_{c,w_1\cdots w_n}
        $$
        \item \label{property 4} Let $v_1,\ldots, v_n\in V$ and $w,v^1,\ldots, v^n\in V^*$ where $w$ is of length $m$. Assume $v = v_1\cdots v_n$ is a context. Fix functions $f\colon m_{v^i}\rightarrow m_{v_i}$ for $i\leq n$ and $g\colon w\rightarrow b$. We assume $vRw$. Denote the function $\theta\colon [m]\rightarrow [n]$ where $w = v_{\theta(1)}\cdots v_{\theta(m)}$. Then 
        $$
        (g\circ (f_{\theta(1)},\ldots, f_{\theta(m)}))*m_{c,v^{\theta(1)}\cdots v^{\theta(m)}}= ((g*m_{v,w})\circ (f_1,\ldots, f_n)) *m_{c,v^1\cdots v^n}
        $$
    \end{enumerate}
\end{lemma}
\begin{proof}
    All parts follow directly from the corresponding axiom of $\Delta_R$-action.  So we will show only the last assertion:
    Denote the length of $v^i$ by $k_i\in\N$. Now $m_{v,w} = \theta_{m_v,*}$ and $m_{c,v^1\cdots v^n} = \phi_*$ for some $\theta\colon [m]\rightarrow [n]$ and $\phi\colon[k_1+\ldots + k_n]\rightarrow [l]$, where $m,n$ and $l$ are the lengths of $w,v$ and $c$, respectively. Consider the following:
    \begin{align*}
        (g*m_{v,w})\circ (f_1,\ldots, f_n)*m_{c,v^1\cdots v^n}
        &= \phi_{m_c,*}(\theta_{m_v,*}(g)\circ (f_1,\ldots, f_n))\\
        &= \phi_{m_c,*}((\theta'_{k_1,\ldots, k_n})_{v^1\cdots v^n,*}(g\circ (f_{\theta(1)},\ldots, f_{\theta(m)})))\\
        &=(\phi\circ\theta'_{k_1,\ldots, k_n})_{m_c,*}(g\circ(f_{\theta(1)},\ldots, f_{\theta(m)}))\\
        &= g\circ (f_{\theta(1)},\ldots, f_{\theta(m))})*m_{c,v^{\theta(1)}\cdots v^{\theta(m)}}.
    \end{align*}
\end{proof}

\begin{definition}[Satisfyiability]
Let $\sigma$ be a signature with a modelable context structure $R$. Let $m$ be a $\sigma$-model in a $\Delta_R$–multicategory $C$. Let $t$ be a term with an $R$-context $v$. We denote $m_v = m(s_1)\cdots m(s_n)$ where $v = x_1\ldots x_n$ and $x_i\colon s_i$ for $i\leq n$. We are ready to define the multi-morphism $m_v(t)\colon m_v\rightarrow m(b)$ for a term $t\colon b$ with an $R$-context $v$:
$$
m_v(t)= \begin{cases}
     id_b*m_{v,v_i},\text{ if $t = v_i$},\\
    m(c)*m_{v,()}, \text{ if $t$ is a constant symbol $c$},\\
    (m(f)\circ (m_{w_1}(t_1),\ldots, m_{w_k}(t_k)))*m_{v,w_1\cdots w_k}, \text{ if } t = f(t_1,\ldots, t_k),
\end{cases}
$$
where $w_i$ is an $R$-context for $t_i,i\leq n$ and $vRw_1\cdots w_n$.

The model $m$ is said to satisfy an $R$-equation $t_1\approx_v t_2$ if $m_v(t_1) = m_v(t_2)$. This is denoted $m\vDash t_1\approx_v t_2$. We say that $m$ satisfies $R$-theory $E$ if $m$ satisfies all the equations in $E$ and we denote this by $m\vDash E$.
\end{definition}

We defined $m_v(f(t_1,\ldots, t_n)) = m(f)(m_{w_1}(t_1),\cdots m_{w_n}(t_n))*m_{v,w_1\cdots w_n}$ where $vRw_1\cdots w_n$ and $w_iR\tau(t_i),i\leq n$. Notice that some contexts $w_1,\ldots, w_n$ satisfying the conditions do exist, since $R$ is a modelable context structure. Furthermore, this definition is independent of the choice of the contexts, since we have assumed the existence of terminal contexts for modelable context structures $R$. In the following lemma, we show the independence in more detail:

\begin{lemma}\label{for soundness}
    Let $\sigma$ be a signature with a modelable context structure $R$. Let $m$ and be a $\sigma$-model in a $\Delta_R$-mulitcategory $C$. Then the following holds
    \begin{enumerate}
        \item Let $t\colon b$ be a term with $R$-contexts $v$ and $w$ where $vRw$. The multimorphism $m_v(t)\colon m_v\rightarrow m(b)$ is well defined and $m_w(t)*m_{v,w} = m_v(t).$
        
        \item Let $v,w$ be $R$-contexts for terms $t_1$ and $t_2$ and assume that $vRw$. Then $m_w(t_1) = m_w(t_2)$ implies that $m_v(t_1) = m_v(t_2)$. 
        
        \item Let $t\colon b$ be a term with an $R$-context $v$. Assume that $(s,w,(w_i)_{i\leq n})$ be an $R$-renaming of variables in $v$. Then
        $$
        m_w(s(t)) = m_v(t)\circ (m_{w_1}(s(v_1)),\ldots, m_{w_n}(s(v_n)))* m_{w,w_1\ldots w_n}
        $$
        where $w_i = \tau(s(v_i))$ for $i\leq n$
    \end{enumerate}
\end{lemma}

\begin{proof}\noindent
    \begin{enumerate}
        \item We can prove both claims simultaneously via induction:
        The claim clearly holds for constant and variable terms $t$. Assume that $t = f(t_1,\ldots, t_n)$ and the claim holds for $t_i,i\leq n$. Let $p_i$ be an $R$-terminal context for $t_i$ for $i\leq n$. Notice then that
        \begin{align*}
            &m(f)\circ (m_{v_1}(t_1),\ldots, m_{v_n}(t_n))*m_{v,v_1\cdots v_n}\\
            &= m(f) \circ (m_{p_1}(t_1)*m_{v_1,p_1},\ldots, m_{p_n}(t_n)*m_{v_n,p_n})*m_{v,v_1\cdots v_n}\\
            &= m(f)\circ (m_{p_1}(t_1),\cdots, m_{p_n}(t_n))*m_{v,p_1\ldots p_n}
        \end{align*}
        Since the morphism $m_v(t)$ is independent of the choice of contexts $v_1,\ldots, v_n$, it follows that the morphism $m_v(t)$ is uniquely defined. In the same induction, we notice that
        \begin{align*}
        &m_w(t)*m_{v,w}\\ 
        &= m(f)\circ (m_{p_1}(t_1),\ldots m_{p_n}(t_n))*m_{w,p_1\cdots p_n})*m_{v,w}\\
        &= m(f)\circ (m_{p_1}(t_1),\ldots, m_{p_n}(t_n))*m_{v,p_1\cdots p_n}\\
        &= m_w(t).
        \end{align*}

        \item Assume that $m_w(t_1) = m_w(t_2)$. Now
        \begin{align*}
            m_v(t_1) 
            &= m_w(t_1)*m_{v,w}\\ 
            &= m_w(t_2)*m_{v,w}\\ 
            &= m_w(t_2).
        \end{align*}
        \item We will show by induction on the structure of $t$ that
        $$
        m_w(s(t)) = m_v(t)\circ (m_{w_1}(s(v_1)),\ldots, m_{w_n}(s(v_n)))* m_{w,w_1\ldots w_n}.
        $$
        If $t = c$ is a constant, then the left-hand side is $m_w(s(c)) = m(c)*m_{w,()}$ and the right-hand side is 
        \begin{align*}
        m_v(c)\circ (m_{w_1}(s(v_1)),\ldots, m_{w_n}(s(v_n)))
        &= !_{(),b,*}(m(c))\circ (m_{w_1}(s(v_1)),\ldots, m_{w_n}(s(v_n)))\\
        &=!'_{(),b,*}(m(c))\\
        &=m(c)*m_{w,()}
        \end{align*}
        Assume that $t = v_i$. The left-hand side becomes $m_w(s(t)) = m_{w}(s(v_i))$. Denote the function $1\mapsto i\colon [1]\rightarrow [n]$ as $\phi$. Now the right-hand side is
        \begin{align*}
            &m_v(v_i)\circ (m_{w_1}(s(v_1)),\ldots, m_{w_n}(s(v_n)))*m_{w,w_1\cdots w_n}\\
            &= (id*m_{v,v_i})\circ (m_{w_1}(s(v_1)),\ldots, m_{w_n}(s(v_n)))*m_{w,w_1\cdots w_n}\\
            &= (id\circ m_{w_i}(s(v_i)))*m_{w,w_i}\\ 
            &= m_w(s(v_i)).
        \end{align*}
        Lastly, assume that $t = f(t_1,\ldots, t_k)$. Since $v$ is a context for $t$, it follows from the assumption that $R$ is a modelable context structure that $vRv^1\cdots v^k$ where $v^i = v^i_1\cdots v^i_{\alpha_i}, v^i_{\alpha_i}\in V,i\leq k$, is an $R$-terminal context of $t_i$ for $i\leq k$. We notate $w^i_j = w_{j'}$ for the unique $j'$ where $v^i_j = v_{j'}$ for $i\leq k$ and $j\leq \alpha_i$. Notice that since $vRv^1_1\cdots v^k_{\alpha_k}$ and $wRw_1\cdots w_n$, it follows that $wRw^1_1\cdots w^k_{m_k}$. Thus there exists an $R$-terminal context $w^i$ for $w^i_1\cdots w^i_{\alpha_i}$ for $i\leq k$ where $wRw^1\cdots w^k$. Thus restrictions of $s$ define $R$-renamings of variables $(s\mid, (w^i_j)_{j\leq \alpha_i},w^i)$ for variables in $v^i$. Thus we may apply the induction hypothesis on terms $t_i$ for $i\leq n$. Now
        \begin{align*}
            &m_w(s(f(t_1,\ldots, t_k))) \\
            &= m_w(f(s(t_1),\ldots, s(t_k)))\\
            &= m(f)\circ (m_{w^1}(s(t_1)),\ldots, m_{w^k}(s(t_k)))*m_{w,w^1\cdots w^k}\\
            &= m(f)\circ (m_{v^1}(t_1)\circ (m_{w^1_1}(s(v_1^1)),\cdots, m_{w^1_{\alpha_1}}(s(v^1_{\alpha_1})))*m_{w^1,w^1_1\cdots
            w^1_{\alpha_1}},\ldots ,\\
            &m_{v^k}(t_k)\circ (m_{w^k_1}(s(v^k_1)),\ldots, m_{w^k_{\alpha_k}}(s(v_{\alpha_k}^k)))*m_{w^k,w^k_1\cdots w^k_{\alpha_k}})*m_{w,w^1\cdots w^k}\\
            &=m(f)\circ (m_{v^1}(t_1)\circ (m_{w^1_1}(s(v_1^1)),\cdots, m_{w^1_{\alpha_1}}(s(v^1_{\alpha_1}))),\ldots ,\\
            &m_{v^k}(t_k)\circ (m_{w^k_1}(s(v^k_1)),\ldots, m_{w^k_{\alpha_k}}(s(v_{\alpha_k}^k))))*m_{w,w^1_1\cdots w^k_{\alpha_k}}\\
            &= (m(f)\circ (m_{v^1}(t_1),\ldots, m_{v^k}(t_k)))\circ (m_{w^1_1}(s(v_1^1)),\ldots, m_{w^k_{\alpha_k}}(s(v_{\alpha_k}^k)))*m_{w,w^1_1\ldots w^k_{\alpha_k}}\\
            &= (m(f)\circ (m_{v^1}(t_1),\ldots, m_{v^k}(t_k))*m_{v,v^1\ldots v^k})\circ (m_{w_1}(s(v_1)),\ldots, m_{w_n}(s(v_n)))*m_{w,w_1\cdots w_n}\\
            &= m_v(t)\circ (m_{w_1}(s(v_1)),\ldots, m_{w_n}(s(v_n)))*m_{w,w_1\cdots w_n}.
        \end{align*}
        The second to last equation follows from the Lemma \ref{Properities of canonical morphisms} (\ref{property 4}).
    \end{enumerate}
\end{proof}

\begin{theorem}[Soundness]\label{Soundness theorem}
Let $\sigma$ be a signature with a modelable context structure $R$. Consider the associated structure category $\Delta = \Delta_R$. Let $E\cup\{\phi\}$ be an $R$-theory. If $E\vdash_R\phi$, then $E\vDash_C\phi$ for any $\Delta$-multicategory $C$. 
\end{theorem}
\begin{proof}
    Let $m\vDash E$ in a $\Delta$-multicategory $C$. Let $T$ be the set $R$-equations $m$ satisfies. It needs to be shown that $D_E\subset T$. We show this by induction. It is clear that $T$ contains $E$, and satisfies reflexivity, symmetry and transitivity. We need to see that $T$ is closed under $R$-substitution: Let $t_1\approx_v t_2\in T$, where $v = v_1\cdots v_n$ and $v_i\in V$ for $i\leq n,$ and assume that $(s_1,w,(w_i)_{i\leq n})$ and $(s_2,w,(w_i)_{i\leq n})$ are $R$-renamings for $v$. Now by Lemma \ref{for soundness}
    \begin{align*}
        m_w(s_1(t_1)) 
        &= m_v(t_1)\circ (m_{w_1}(s_1(v_1)),\ldots, m_{w_n}(s_1(v_n)))* m_{w,w_1\ldots w_n}\\
        &= m_v(t_2)\circ (m_{w_1}(s_2(v_1)), \ldots, m_{w_n}(s_2(v_n)))* m_{w,w_1\ldots w_n}\\
        &= m_w(s_2(t_2))
    \end{align*}
    Thus $s_1(t_1)\approx_w s_2(t_2)\in T$. Therefore $D_E\subset T$.
\end{proof}

\subsection{Completeness in Sets}
We showed that for a given signature $\sigma$ with a modelable context structure $R$ it holds for any $R$-theory $E$ that $D^R_E\subset \text{True}(m)$ for all $m$ models in any $\Delta_R$-multicategory $C$. We will show a strong version of the converse that there is a single universal model $m$ for $E$ in a $\Delta_R$-multicategory $C$ such that $\text{True}(m) = D^R_E$. However, this universal model does not live in the cartesian multicategory of sets. In this subsection, we are going to show that if $R$ is balanced-complete or the cartesian context structure, then $D^R_E$ is exactly the intersection of $\text{True(m)}$ for all models $m$ in the cartesian multicategory of sets. This has computationally an interesting corollary, that to show $E\vdash \phi$ it suffices to find a balanced-complete $R$ such that $E\vdash_R \phi$.

\begin{definition}[Term-Algebra]
    Let $\sigma = (S,M,V)$ be a signature with a modelable context structure $R$. Let $E$ be an $R$-theory. Denote by $C_v$ and $C_t$ the set of $R$-context over a word $v\in V^*$ and a $\sigma$-term $t$, respectively. Assume the first condition of balanced complete context structures for $R$: If $C_{v_i}\subset C_{w_i}$ for $i\leq n$, then $C_{v_1\cdots v^n}\subset C_{w_1\cdots w_n}$. We define four different term-algebras $n_i,m_i$ for $i = 1,2$ in the cartesian multicategory $\textbf{Set}$ of sets:
    \begin{itemize}
        \item Define $n_1(s) = \{(v,t)\mid t\colon s \text{ is a $\sigma$–term, $v\in V^*$ and $C_v\subset C_t$}\}$ and $n_2(s) = \{t\mid t\colon s \text{ is a $\sigma$–term}\}$ for a $\sigma$-sort $s$.
        \item Let $f\colon a\rightarrow b\in M$, where $a = a_1\cdots a_n$ and $a_i\in S$ for $i\leq n$. We set $n_i(f)\colon n_i(a)\ra n_i(b)$
        {\small
        $$
        n_1(f)((v_1,t_1),\ldots, (v_n,t_n)) = (v_1\cdots v_n,f(t_1,\ldots, t_n))\text{ and }n_2(f)(t_1,\ldots, t_n) = f(t_1,\ldots, t_n)
        $$
        }
        for $(v_j,t_j)\in n_1(a_i), j\leq n$ and $i = 1,2$. The function $n_1(f)$ is well-defined by the modelability of $R$.
    
    \item Let $s\in S$. Set $m_i(s) = n_i(s)/\sim_s^i$, where $\sim_s^i$ is called the provability relation on $n_i(s)$, where $(v_1,t_1)\sim^1(v_2,t_2)$ if and only if $C_{v_1} = C_{v_2}$ and $E\vdash_R t_1\approx_w t_2$ for all $w\in C_{v_1}$ and $t_1\sim^2 t_2$ if and only $E\vdash_{\text{Cart}} t_1\approx_w t_2$ for some context $w$ containing variables in $t_1$ and $t_2$. 
    \item We define $m_i(f)\colon m_i(a)\ra m_i(b)$ as the unique function making the diagram
    $$
    \begin{tikzcd}
n_i(a) \arrow[r, "n_i(f)"] \arrow[d, "q_i(a)"'] & n_i(b) \arrow[d, "q_i(b)"] \\
m_i(a) \arrow[r, "m_i(f)"']                   & m_i(b)                  
\end{tikzcd}
    $$
    commute for a morphism symbol $f\colon a\rightarrow b$. Here we denote the quotient map $n_i(s)\rightarrow m_i(s)$ as $q_i(s)$ for sorts $s$.
    \end{itemize}
    The term algebras $n_1,n_2,m_1,m_2$ are called balanced $R$-term algebra, $R$-term algebra, balanced $E$-term algebra and $E$-term algebra. The modelability of $R$ guarantees that $m_1$ is well-defined.
\end{definition}

\begin{lemma}\label{substitution}
    Let $\sigma = (S,M,V)$ be a signature with a modelable context structure $R$. Let $n_1,n_2$ be the balanced $R$-term model and the $R$-term model, respectively. Let $t$ be a term with a cartesian context $v=v_1\ldots v_n$, where $v_i\in V$ for $i\leq n$. Consider the function $\theta\colon[m]\rightarrow [n]$ where $\tau(t) = v_{\theta(1)}\cdots v_{\theta(m)}$. Let $u = ((v^1, u_1),\ldots, (v^n,u_n))\in n_v$. Denote the renaming of variables $v_i\xmapsto{s} u_i$, for $i\leq n$ in $v$. Then
    $$
    n_{1,v}(t)(u) = (v^\theta\coloneqq v^{\theta(1)}\cdots v^{\theta(m)},s(t)) \text{ and }n_{2,v}(t)(u_1,\ldots, u_n) = s(t).
    $$
\end{lemma}
\begin{proof}
    The case for $n_2$ is similar, so we prove only the case for $n_1 = n$. We prove the claim by induction on the structure of the term $t$. If $t = c$ is a constant, then the right-hand side is $((),c)$.  The left-hand side is 
    $$
    n_v(c)(u) = ((),c).
    $$
    Assume that $t = v_i$. The right-hand side is $(v^i,u_i)$. The left-hand side is 
    $$
    n_v(v_i)((v_1,u_1),\ldots, (v_n,u_n)) = (v^i, u_i)
    $$

    Assume then that $t = f(t_1,\ldots, t_k)$ and the claim holds for $t_1,\ldots, t_k$. Let $\theta_i\colon [m_i]\rightarrow [n]$ be the unique function where $\tau(t_i) = v^{\theta_i}$ for $i\leq k$. Notice that $v^\theta = v^{\theta_1}\cdots v^{\theta_k}$. Now we attain
        \begin{align*}
        n_v(t)(u)
        &= n(f)(n_v(t_1)(u),\ldots, n_v(t_k)(u)) \\
        &= n(f)((v^{\theta_1},s(t)),\ldots, (v^{\theta_k},s(t_k))) \\
        &= (v^{\theta_1}\cdots v^{\theta_k},f(s(t_1),\ldots, s(t_k))) \\
        &= (v^{\theta},s(t)).
        \end{align*}
\end{proof}

\begin{lemma}[Term-naturality]\label{Term-naturality}
    Let $\sigma$ be a signature with a modelable context structure $R$ and let $f\colon m\ra n$ be a $\sigma$-model morphism in a $\Delta_R$-multicategory $C$. Let $t\colon b$ be a term with an $R$-context $v$. Then
    $$
    \begin{tikzcd}
m_v \arrow[r, "m_v(t)"] \arrow[d, "f^v"'] & m(b) \arrow[d, "f(b)"] \\
n_v \arrow[r, "n_v(t)"']                  & n(b)                  
\end{tikzcd}
    $$
    commutes.
\end{lemma}

\begin{proof}
    If $t = c$ is a constant, then 
    \begin{align*}
    f(b)\circ m_v(c) 
    &= f(b)\circ (m(c)*m_{v,()})\\
    &= (f(b)\circ m(c))*m_{v,()} \\
    &= n(c)*m_{v,()} \\
    &= (n(c)*n_{v,()})\circ f^v \\
    &= n_v(c)\circ f^v
    \end{align*}
    If $t = v_i$ is a variable symbol, then denote $1\xmapsto{\phi} i\colon [1]\rightarrow [n]$ and now
    $$
    f(b)\circ m_v(v_i) = f(b)\circ \phi_*(id_b) = \phi_*(f(b)\circ id_b) = \phi_*(f(b))
    $$
    and
    $$
    n_v(v_i)\circ f^v = (\phi_*(id_b))\circ f^v = (\phi'_{1,\ldots, 1})_*(id_b\circ f(b)) = \phi_*(f(b)).
    $$
    Assume then that $t = \alpha(t_1,\ldots, t_k)$ and the induction hypothesis holds for $t_1\colon a_1,\ldots, t_k\colon a_k$. Let $v^i$ be an $R$-terminal context for $t_i$ for $i\leq k$. Now
    \begin{align*}
        f(b)\circ m_v(\alpha(t_1,\ldots, t_k))
        &= f(b)\circ m_{v,v^1\cdots v^n}(m(\alpha)\circ (m_{v^1}(t_1),\ldots, m_{v^k}(t_n)))\\
        &= m_{v,v^1\cdots v^k}(f(b)\circ m(\alpha)\circ (m_{v^1}(t_1),\ldots, m_{v_k}(t_k)))\\
        &= m_{v,v^1\cdots v^k}(n(\alpha)\circ (f(a_1)\circ m_{v^1}(t_1),\ldots, f(a_k)\circ m_{v_k}(t_k)))\\
        &= m_{v,v^1\cdots v^k}(n(\alpha)\circ (n_{v^1}(t_1)\circ f^{v^1},\ldots, n_{v^k}(t_k)f^{v^k}))\\
        &= m_{v,v^1\cdots v^k}((n(\alpha)\circ (n_{v^1}(t_1),\cdots n_{v^k}(t_k))\circ f^{v^1\cdots v^k})\\
        &= m_{v,v^1\cdots v^k}(n(\alpha)\circ (n_{v^1}(t_1),\ldots, n_{v^k}(t_k))\circ f^{v^1\ldots, v^k})\\
        &= n_{v,v^1\cdots v^k}(n(\alpha)\circ (n_{v^1}(t_1),\ldots, n_{v^k}(t_k)))\circ f^v\\
        &= n_v(t)\circ f^v.
    \end{align*}
\end{proof}

\begin{theorem}\label{Provability relation}
    Let $\sigma$ be a signature with a modelable context structure $R$. Let $E$ be an $R$-theory. Denote by $C_v$ the set of contexts over a word $v\in V^*$. Then the following holds:
    \begin{enumerate}
        \item Let $n$ and $\sim$ be the balanced $R$-term algebra the balanced $E$-provability relation on $n$, respectively. Denote by $m = n/\sim$ the balanced $E$-term model. Let $t_1\approx_v t_2$ be an $R$-equation. Then $m\vDash t_1\approx_v t_2$ if and only if $E\vdash_R t_1\approx_w t_2$ for some $R$-terminal context $w$ for both $t_1$ and $t_2$. 
        \item Let $n$ and $\sim$ be the $R$-term algebra and the $E$-provability relation on $n$, respectively. Let $m$ be the $E$-term model. Let $t_1\approx_v t_2$ be a $\sigma$–equation. Then $m\vDash t_1\approx_v t_2$ if and only if $t_1\sim t_2$. 
    \end{enumerate}
\end{theorem}
\begin{proof}\hfill
    \begin{enumerate}
        \item Assume first $m_v(t_1) = m_v(t_2)$ for the forwards direction. Now by Lemmas \ref{substitution} and \ref{Term-naturality}
    \begin{align*}
        [(\tau(t_1),t_1)]
        &= [n_v(t_1)((v_1,v_1),\ldots,(v_n,v_n))]\\
        &= m_v(t_1)([v_1,v_1],\ldots, [v_n,v_n])\\
        &= m_v(t_2)([v_1,v_1],\ldots, [v_n,v_n])\\
        &= [n_v(t_2)((v_1,v_1),\ldots, (v_n,v_n)]\\
        &= [(\tau(t_2),t_2)]
    \end{align*}
    Therefore, $C_{t_1} = C_{t_2}$ and $E\vdash_R t_1\approx_w t_2$ for all $w\in C_{t_1}$. Since $C_{t_1} = C_{t_2}$ and $v\in C_{t_1}$, it follows that $t_1$ and $t_2$ have the same variables appearing and hence they have a common terminal context $w$ and therefore $E\vdash_R t_1\approx_w t_2$.

    For the converse, we may assume that $E\vdash_R t_1\approx_v t_2$ and that $t_1\approx_v t_2$ is a balanced equation. Let $((v^1,u_1),\ldots, (v^n,u_n))\in n_v$. Consider the function $\theta_i$, where $\tau(t_i) = v_{\theta_i}$ for $i = 1,2$. We show that $m_v(t_1)([v^1,u_1],\ldots, [v^n,u_n]) = m_v(t_2)([v^1,u_1],\ldots, [v^n, u_n])$. Let $s$ be the renaming $v_i\xmapsto{s} u_i$ for $i\leq n$. Notice that 
    $$
    m_v(t_i)([v^1,u_1],\ldots, [v^n,u_n]) = [v^{\theta_i},s(t_i)] = [v^1\cdots v^n,s(t_i)],
    $$
    since $C_{v^{\theta_i}} = C_{v^1\cdots v^n}$ by the Corollary \ref{Sets of contexts corollary}. Let $wRv^1\cdots v^n$. Suffices to show that $E\vdash_R s(t_1)\approx_w s(t_2)$. Since $R$ is modelable, there exists contexts $c_i$ where $c_iRv^i$ for $i\leq n$ and $wRc_1\cdots c_n$. Thus $E\vdash_R s(v_i) = u_i\approx_{c_i} u_i = s(v_i)$ for $i\leq n$ and $E\vdash_R t_1\approx_v t_2$. Thus $E\vdash_R s(t_1)\approx_w s(t_2)$ by the definition of $R$-deduction. Hence $m_v(t_1) = m_v(t_2)$. 

    \item The proof is similar to the previous part and hence omitted.
    \end{enumerate}
\end{proof}
\begin{theorem}[$\textbf{Set}$-Completeness]\label{Completeness Theorem}
    Let $\sigma = (S,M,V)$ be a signature with a modelable context structure $R$. Let $E\cup\{\phi\}$ be an $R$-theory. Assume one of the following assertions:
    \begin{enumerate}
        \item The context structure $R$ is balanced, meaning that $cRv$ implies that $\text{Var}(c) = \text{Var}(v)$.
        \item The context structure $R$ is the cartesian context structure.
    \end{enumerate}
    Then
    $$
    E\vdash_R \phi \text{ if and only if } E\vDash_{\textbf{Set}} \phi
    $$
\end{theorem}
\begin{proof}
    Soundness already gives us that $E\vdash_R \phi$ implies $E\vDash_{\textbf{Set}}\phi$. 
    \begin{enumerate}
        \item Assume that $R$ is balanced-complete. Hence the equation $\phi$ is balanced and the Theorem \ref{Provability relation}(1) implies the completeness.
        
        \item Assume that $R$ is the cartesian context structure and that $E\not\vdash \phi$. Denote $\phi = t_1\approx_v t_2$. Consider a new signature $\sigma' = (S',M',V')$ and a $\sigma'$-theory $E'$, where
        \begin{align*}
            S' &= \{s\in S\mid \text{$\text{Var}(t)\subset\text{Var}(v)$ for some $\sigma$-term $t\colon s$}\},\\
            M' &= \{f\colon a\rightarrow b\in M \mid  a\in S'^*, b\in S\},\\
            V' &= \{v\in V\mid  v\colon s\text{ for some $s\in S'$}\}\text{ and }\\
            E' &= \{t_1\approx_w t_2\in E \mid  w\in V'^*\}
        \end{align*}
        Notice that the codomain of a function symbol in $M'$ is still in $S'$ which follows from how one constructs terms. Let $m' = m_{E'}$ be the $E'$-term model. We extend $m'$ to be a $\sigma$-model $m$ by setting 
    $$
    m(s) = \begin{cases}
        m'(s),\text{ if } s\in S'\\
        \emptyset, \text{ if } s\not\in S'
    \end{cases}\text{ and }
    m(f) = \begin{cases}
        m'(f)\colon m(a)\rightarrow m(b),\text{ if } f\in M'\\
        \emptyset\colon m(a)\rightarrow m(b),\text{ if } f\not\in M'
    \end{cases}
    $$
    Since $m_v(t) = m'_v(t)$ for any $\sigma'$ term $t$ in a context $v\in V'^*$, it follows that $m\vDash \psi$ for $\psi\in E'$. The model $m$ satisfies the rest of the equations in $E$ by vacuity, since $m_v$ is the empty set for any $v\in V\setminus V'$. Therefore $m\vDash E$. Assume towards a contradiction that $m\vDash t_1\approx_v t_2$. Thus $m'\vDash t_1\approx_v t_2$ and by Theorem \ref{Provability relation}(2) $t_1\sim t_2$, where $\sim$ is the provability relation. Thus $E'\vdash t_1\approx_w t_2$ for some context $w\in V'^*$ for $t_1$ and $t_2$. We define a renaming of variables in $w$ as follows
    $$
    x\xmapsto{s}\begin{cases}
        x,\text{ if } x\text{ is expressed in } v\\
        t_x,\text{else},
    \end{cases}
    $$
    where $t_x$ is a term with the same type as $x$ and $t_x$ has its variables expressed in $v$. Notice that $v$ is a context for $\tau(s(w_1))\cdots \tau(s(w_n))$. Therefore, the substitution rule of deduction yields $E\vdash t_1\approx_v t_2$, which is a contradiction. Therefore, $m\not\vDash t_1\approx_v t_2$. \qedhere
    \end{enumerate}
\end{proof}

\begin{theorem}[Multi-Categorical Meta-Theorem]
    Let $\sigma$ be a signature with a modelable context structure $R$. Assume that $R$ is either balanced or the cartesian context structure. Consider the associated structure category $\Delta$. Let $E\cup\{\phi\}$ be an $R$-theory. Let $C$ be a $\Delta$-multicategory. Then
    $$
    E\vDash_{\textbf{Set}} \phi \text{ implies } E\vDash_C \phi.
    $$
\end{theorem}
\begin{proof}
    If $E\vDash_{\textbf{Set}} \phi$, then by the Completeness Theorem \ref{Completeness Theorem} $E\vdash_R \phi$, and by the Soundness Theorem \ref{Soundness theorem} $E\vDash_C \phi$.
\end{proof}
The Multicategorical Meta-Theorem applies for $\Delta$-multicategories, where $\Delta$ is one of the following $6$ structures:
\begin{multicols}{2}
    \begin{enumerate}
        \item $\Delta_{\{1\}}$ of identities.
        \item $\Delta^{\{1\}}$ of bijections.
        \item $\Delta_\N = \Delta^\N$ of all functions.
        \item $\Delta^{\N_{1}}$ of surjections.
        \item $\Psi_{N_1}$ of left surjections.
        \item $\Psi^{N_1}$ of right surjections.
    \end{enumerate}
\end{multicols}
The structure categories $\Delta_{\{0,1\}}$ and $\Delta^{\{0,1\}}$ corresponding to strictly increasing maps and injections are do not have an associated completeness theorem concerning sets. Consider the signature $\sigma = (S = \{*\}, M=\{f\colon **\rightarrow *\})$ with a theory $E = \{f(x,y)\approx_{xyz} x)\}$. Now $E\vdash f(x,y)\approx_{xy} x$, but $E\not\vdash_R f(x,y)\approx_{xy} x$ where $R$ is either $R^{\{0,1\}}$ or $R_{\{0,1\}}$. Let $D = \{t_1\approx_v t_2\mid \text{if $\text{Var}(t_i) = \text{Var}(v)$ for either $i = 1,2$, then $t_1 = t_2$}\}$ be a family of $R$-equations. The family $D$ satisfies the closure properties of $R$-deduction and $E\subset D$ and thus $D_E^R\subset D$. Since $f(x,y)\approx_{xy} x\not\in D$ it follows that $E\not\vdash_R f(x,y)\approx_{xy} x$.
\section{Categorical completeness}
Consider a modelable context structure $R$. At the end of the previous section, we saw that two of the eight modelable deduction systems $\vdash_R$ do not have complete semantics in sets. In this section, we remedy this problem and show that $\vdash_R$ has complete semantics if we permit all the models in all $\Delta_R$-multicategories, not just the ones in sets.

For each $R$-theory $E$, we construct the initial $\Delta_R$-multicategory equipped with model $m\vDash E$. We show that for $m$ it holds that $D_E^R = \text{True(m)}$. This result will be called categorical completeness.

The following theorem is known and very useful as it allows to construct free algebras.
\begin{theorem}
    Let $\sigma = (S,M,V)$ be a signature and let $E$ be a $\sigma$-theory. Then the forgetful functor $U\colon \text{Model}(\textbf{Set}, E)\rightarrow \textbf{Set}^S$ has a left adjoint.
\end{theorem}
\begin{proof}
    Let $A = (A_s)_{s\in S}$ be a family of sets. We define the typed set $A\text{-Term}\rightarrow S$ of $A$-terms as follows:
    \begin{itemize}
        \item  $a\colon s\in A\text{-Term}$ for $a\in A_s$.
        \item  $c:s\in A\text{-Term}$ for $c:s\in M$.
        \item $f(t_1,\ldots, t_n)\colon b\in A\text{-Term}$ for $f\colon a_0\cdots a_n\rightarrow b\in M$ and $t_0\colon a_0,\ldots, t_n\colon a_n\in A\text{-Term}$. 
    \end{itemize}
    We define $A$-term model $n$ as follows:
    \begin{itemize}
        \item Set $n(s) = \{t\colon s\in A\text{-Term}\}$.
        \item For $f\colon a_1\cdots a_n\rightarrow b\in M$ we define
        $$
            n(f)\colon n(a_1\cdots a_n)\rightarrow n(b), (t_1,\ldots, t_n)\mapsto f(t_1,\ldots, t_n).
        $$
    \end{itemize}
    We define an $E$-model out of $n$ via a $\sigma$-congruence $\sim$ generated by the pairs of elements $(n_v(t_1)(x),  n_v(t_2)(x))$ for all $E\vdash t_1\approx_v t_2$ and $x\in n_v$. As with the provability relation, $n/\sim$ satisfies the equations in $E$. Furthermore, given any family of functions $f = (A_s\rightarrow p_s)_s$ where $p\vDash E$, $f$ extends uniquely to a $\sigma$-morphism $\overline{f}\colon n\rightarrow p$. Since $p$ satisfies all the equations deduced from $E$, it follows that $\overline{f}$ extends uniquely along the quotient to a $\sigma$-morphism $m\rightarrow p$. This proves the claim.
\end{proof}
Since left adjoints preserve initial objects, the initial $E$-model $m$ consists of the equivalence classes of pure $\sigma$-terms $t$, which are $\sigma$-terms where no variable is expressed.

\begin{definition}[Universal model]
    Let $\sigma$ be a signature with a modelable context structure $R$. Consider the signature extended signature $\Sigma$ of $\sigma$ and $R$, defined as follows:
    \begin{itemize}
        \item The set of $\Sigma$-sorts is $S' = S^*\times S$.
        \item For each $b,b_1,\ldots, b_n, c\in S, a,a^1,\ldots, a^n\in S^*, f\colon a\rightarrow b\in M$ and $\theta\colon [m]\rightarrow [n]$ in $\Delta_R$, we define $\Sigma$–morphism symbols
        \begin{align*}
        \circ_{a^1,\ldots, a^n, b_1,\ldots, b_n, c}
        &\colon (b_1\cdots b_n,c)(a^1,b_1)\cdots (a^n,b_n)\rightarrow (a^1\cdots a^n, c)\\
        id_c&\colon (c,c)\\
        \theta_{*,b,c}&\colon (b_{\theta(1)}\cdots b_{\theta(m)},c)\rightarrow (b_1\cdots b_n,c)\\
        f&\colon (a,b).
        \end{align*}
        \end{itemize}

        The symbols $\circ, id, \theta_{*,a,b}$ and $f\in M$ are called composition symbols, identities symbols, $\Delta$-action symbols and the internalized $\sigma$-symbols, respectively.  We will often suppress the indices of the function symbols. Consider the following theories:
        \begin{itemize}
            \item We define a linear $\Sigma$-theory $\overline{\sigma}$ called the categorization of $\sigma$ defined as follows:
            \begin{align*}
                \circ(\circ (h,g_1,\ldots, g_n), f^1_1,\ldots, f^n_{m_n}) 
                &\approx_{hgf}\circ(h,\circ(g_1,f^1_1\ldots f^1_{m_1}),\ldots, \circ(g_n,f^n_1,\ldots, f^n_{m_n})),\\
                \circ(id_d,h)&\approx_h h,\\
                \circ(h,id_{c_1},\ldots, id_{c_n})&\approx_h h,\\
                (id_{[n]})_{*,c_1\cdots c_n,d}(h) &\approx_h h,\\
                (\phi\theta)_*(h') &\approx_{h'} \phi_*(\theta_*(h')),\\
                \circ(\theta_*(h'),g_1,\ldots, g_n) &\approx_{h'g} (\theta'_{m_1,\ldots, m_n})_*(\circ(h',g_{\theta(1)},\ldots, g_{\theta(m)})),\\
                \circ(h,\theta^1_*(g_1'),\ldots, \theta^n(g_n')) &\approx_{hg'} (\theta^1+\ldots +\theta^n)_*(\circ(h,g'_1,\ldots, g'_n))
            \end{align*}
            for pairwise different variable symbols 
            {\small
            $$
            f^i_j\colon (a^{i,j}, b^i_j), g_i\colon (b^i_1\cdots b^i_{m_i},c_i), g_i'\colon (b^i_{\theta^i(1)},\ldots,b^i_{\theta^i(k_i)},c_i), h\colon (c_1\ldots c_n,d),h'\colon (c_{\theta(1)}\cdots c_{\theta(m)},d)
            $$
            }
            \noindent
            for $a^{i,j}\in S^*$ and $b^i_j,c_i,d\in S$, where $i,n,m_i\in \N$ and $i\leq n$ and $j\leq m_i$ and for functions $\theta\colon[m]\rightarrow [n], \psi\colon [n]\rightarrow [p]$ and $\theta^i\colon [k_i]\rightarrow [m_i],$ for $i\leq n$, in $\Delta_R$. Here $hgf$, $h'g$ and $hg'$ denote the appropriate contexts.

            \item Let $E$ be an $R$-theory. We define the $\Sigma$-theory $\widetilde{E}$ of internalized $E$-theory as the set consisting of equations $N_v(t_1)\approx_{()}N_v(t_2)$ for each $t_1\approx_v t_2\in E$. Since $R$ is a modelable context structure, each word has a terminal context by Theorem \ref{characterizing modelable context-structures}. Choose the terminal contexts $\overline{v}$ for each word $v$. Let $v$ be a context for a term $t\colon b$ and $v = v_1\cdots v_n$ for $v_i\colon a_i\in V, i\leq n$. We can define the pure $\Sigma$-term (has no variables expressed) $N_v(t)\colon (a_1\cdots a_n,b)$ as 
            $$
            N_v(t) = \begin{cases}
                N_{v,(), b}(c),\text{ if $t = c$ is a constant}\\
                N_{v,v_i,b}(id_{a_i}), \text{ if $t = v_i$ for some $i\leq n$}\\
                N_{v,\overline{\tau(t_1)}\cdots\overline{\tau(t_k)},b}(\circ (f,N_{\overline{\tau(t_1)}}(t_1),\ldots N_{\overline{\tau(t_k)}}(t_k))),\text{ if $t = f(t_1,\ldots, t_k)$}
            \end{cases}
            $$
            and $N_{v,w,b}$ is the function symbol $\theta_{*,v,b}$ where $\theta\colon [m]\rightarrow [n]$ is the unique function in $\Delta_R$ where $w = v_{\theta(1)}\cdots v_{\theta(m)}$. 
            We call the theory $\overline{E} = \overline{\sigma}\cup\widetilde{E}$ the multicategorical $E$-theory.
    \end{itemize}
    An $\overline{E}$-model in $\textbf{Set}$ corresponds to a $\Delta_R$-multicategory with $S$ as the set of objects and it is equipped with a choice of an $E$-model. Consider the initial $\overline{E}$–model $C$ in $\textbf{Set}$ and denote by $m$ the $E$-model in $C$.
\end{definition}

\begin{lemma}\label{Functors mapping models}
    Let $\sigma = (S,M,V)$ be a signature with a modelable context structure $R$. Let $T\colon M\rightarrow N$ be a morphism of $\Delta_R$-multicategories. Then $T$ induces a functor 
    $$
    \overline{T}\colon Model(M,\sigma)\rightarrow Model(N,\sigma)
    $$
    from the $\sigma$-models in $M$ to those in $N$. Furthermore, $T(m_v(t)) = \overline{T}(m)_v(t)$ for any term $t$ in an $R$-context $v$. Especially, if $m\vDash_M t_1\approx_v t_2$, then $\overline{T}(m)\vDash_N t_1\approx_v t_2$, and in the case $T$ is faithful, the converse holds.
\end{lemma}
\begin{proof}
    Let $m$ be a $\sigma$-model. We define $\overline{T}(m)_s = T(m_s)$ for $s\in S$ and $\overline{T}(m)(f) = T(m(a))\rightarrow T(m(b))$ for $f\colon a\rightarrow b\in M$. We set $\overline{T}(\alpha)(s)= T(\alpha(s))\rightarrow T(m(s))\rightarrow T(n(s))$ for a $\sigma$–morphism $\alpha\colon m\rightarrow n$ in $M$. Clearly, $\overline{T}$ satisfies the functoriality laws. Let $t$ be a term in an $R$-context $v$. We show that $T(m_v(t)) = \overline{T}(m)_v(t)$. First notice that $T(m_{v,w}) = \overline{T}(m)_{v,w}$, since $T$ respects the $\Delta$–action on the multicategories $M$ and $N$. Now
    \begin{itemize}
        \item If $t = c$ is a constant, then $T(m_v(c)) = T(m_{v,()}(c)) = \overline{T}(m)_{v,()}(c) = \overline{T}(m)_v(c))$. 
        \item If $t = v_i$ is a variable symbol, then $T(m_v(v_i)) = T(m_{v,v_i}(id)) = \overline{T}(m)_{v,v_i}(id)  = \overline{T}(m)_v(v_i).$
        \item Assume that $t = f(t_1,\ldots, t_n)$. Denote by $v_i$ the terminal context of $t_i$ for $i\leq n$. Now
        \begin{align*}
            T(m_v(f(t_1,\ldots, t_n))) 
            &= T(m_{v,v_1\cdots v_n}((m(f)(m_{v_1}(t_1),\ldots, m_{v_n}(t_n)))))\\
            &= \overline{T}(m)_{v,v_1\cdots v_n}(\overline{T}(m)(f)(\overline{T}(m)_{v_1}(t_1),\ldots, \overline{T}(m)_{v_n}(t_n))\\
            &= \overline{T}(m)_v(f(t_1,\ldots, t_n))
        \end{align*}
    \end{itemize}
    Hence $T(m_v(t)) = \overline{T}(m)_v(t)$. Now if $m$ satisfies an $R$-equation $t_1\approx_v t_2$, then 
    $$
    \overline{T}(m)_v(t_1) = T(m_v(t_1)) = T(m_v(t_2)) = \overline{T}(m)_v(t_2).
    $$
    Thus $\overline{T}(m)$ satisfies the equation $t_1\approx_v t_2$. Similarly with the converse, if $T$ is faithful.
\end{proof}
\begin{theorem}[Completeness of the Universal Model]\label{completeness of um}
    Let $E$ be an $R,\sigma$-theory, where $R$ is a modelable context structure and $\sigma = (S,M,V)$. Denote by $\Sigma$ the multicategorically extended signature of $\sigma$. Let $\overline{\sigma}$ be the categorization theory of $\sigma$. Let $\overline{E}$ be the multicategorical $E$-theory. Let $C$ be the initial $\overline{E}$-model and denote by $m$ the $E$-model in $C$. Then the following claims hold:
    \begin{enumerate}
        \item The pair $(C,m)$ defines the initial object in the category of $\Delta_R$–multicategories with a choice of an $E$-model.
        \item It holds that $E\vdash_R \phi$ if and only if $m\vDash \phi$ for any $R$-equation $\phi$. 
  \end{enumerate}
\end{theorem}

\begin{proof}\hfill
    \begin{enumerate}
        \item Let $n\vDash E$ be a model in $\Delta_R$-multicategory $D$. Consider the full submulticategory $C$ of $D$ defined by objects $n_s$ for $s\in S$. The pair $(C,n)$ defines exactly an $\overline{E}$-model. Since the $\Sigma$-morphisms $C\rightarrow N$ correspond exactly with functors $C\rightarrow N$ respecting the models $m$ and $n$, it follows that $(C,m)$ defines the initial $\Delta_R$-multicategory equipped with a model for $E$. 

        \item 
        The direct holds by soundness. We show the completeness. Notice that the initial $\overline{E}$-model $C$ is the quotient of the pure $\Sigma$-terms $A$ according the theory $\overline{E} = \widetilde{E}\cup\overline{\sigma}$. One attains $C$ up to an isomorphism by first quotienting $A$ with the congruence generated by $\overline{\sigma}$ and then quotienting by the congruence generated by $\overline{E}$. Denote by $B$ the quotient of $A$ according to the theory $\overline{\sigma}$. Denote by $b$ the $\sigma$-model in $B$. Consider the family of relations $\sim = (\sim_s)_{s\in S}$, where
        $$
        \sim_s = \{(b_v(t_1),b_v(t_2))\mid E\vdash_R t_1\approx_v t_2\}\text{ for $s\in S$}
        $$
        Notice that $C$ is isomorphic to $B$ quotiented by the congruence generated by $\sim$. It suffices to show that $\sim$ is a $\Sigma$–congruence on $B$. Let $n$ be the $\sigma$–term algebra. Thus there exists a functor $F\colon B\rightarrow\textbf{Set}$ of multicategories mapping the model $b$ to $n$. Thus for any constant $\Sigma$-term $T\colon (a,b)$, we attain a function $\overline{T} = F(B_{()}(T))\colon n(a)\rightarrow n(b)$. Notice that by Lemmas \ref{Functors mapping models} and \ref{substitution}, $\overline{b_v(t)}(v_1,\ldots, v_n) = n_v(t)(v_1,\ldots, v_n) = t$ for any context $v = v_1\cdots v_n$ for a $\sigma$-term $t$ where $v_i\in V,i\leq n$. Let $T\colon (a_1\cdots a_n,b)$ be a pure $\Sigma$-term. Let $v = v_1\cdots v_n$ be a context where $v_i\colon a_i$ for $i\leq n$. Let $t = \overline{T}(v_1,\ldots, v_n)$. By induction one sees that $v$ is an $R$-context for $t$ and $b_v(t) = [T]$, where $[T]$ is the equivalence class of $T$ in $B(a,b)$. Notice that $[b_v(t)] = [b_{v'}(t')]$ if and only if $v$ and $v'$ have the same length, $v_i\xmapsto{s} v_i', i\leq n$, is a renaming of variables $v$ and $t' = s(t)$. The converse is clear. The direct follows from the fact that $n_v(t) = F(b_v(t)) = F(b_{v'}(t)) = n_{v'}(t)$ by Lemma \ref{Functors mapping models}. Now $t' = n_v(t')(v_1',\ldots, v_n') = n_{v}(t)(v_1',\ldots, v_n') = s(t)$ by Lemma $\ref{substitution}$. Thus $t' = s(t)$. Notice that if $v_i\colon a_i, i\leq n$ and $v = v_1\cdots v_n$ is a context and $T_1,T_2\colon (a_1\cdots a_n,b)$ are pure $\Sigma$-terms, then $[T_1]\sim [T_2]$ is equivalent to $E\vdash_R \overline{T_1}(v_1,\ldots, v_n) \approx_v \overline{T_2}(v_1,\ldots, v_n)$. A direct verification then shows that $\sim$ is a $\Sigma$–congruence on $B$.
    \end{enumerate}
\end{proof}

\section{Conclusion}
We showed a bijective correspondence between context structures $R$ on an infinite set $V$ and structure categories $\Delta$. Consider a multi-sorted signature $\sigma= (S,M,V)$ with a context structure $R$. We developed $R$-deduction for $R$-theories and showed that if $R$ is a modelable context structure, then $E\vdash_R \phi$ if and only if $E\vDash_C \phi$ for all $\Delta_R$-multicategories $C$. The structure categories $\Delta_R$ corresponding to modelable context structures are the following:
\begin{multicols}{2}
    \begin{enumerate}
        \item $\Delta_{\{0,1\}}$ of strictly increasing maps.
        \item $\Delta^{\{0,1\}}$ of injections.
        \item $\Delta_{\{1\}}$ of identities.
        \item $\Delta^{\{1\}}$ of bijections.
        \item $\Delta_\N = \Delta^\N$ of all functions.
        \item $\Delta^{\N_{1}}$ of surjections.
        \item $\Psi_{N_1}$ of left surjections.
        \item $\Psi^{N_1}$ of right surjections.
    \end{enumerate}
\end{multicols}
For the latter six of the eight structure categories $\Delta$, we showed that the cartesian multicategory of sets gives complete semantics for $R$-deduction for $R = R_\Delta$ and showed a counter-example for the first two. In other words, $E\vDash_{\textbf{Set}}\phi$ implies $E\vdash_R \phi$ for cartesian context structure $R$ or a balanced and modelable $R$. This has an immediate corollary called the Multicategorical Meta-Theorem:
$$
E\vDash_{\textbf{Set}} \phi \text{ implies }E\vDash_M\phi
$$
for all $\Delta_R$-multicategories $M$. This Multicategorical Meta-Theorem allows the transportation of results from the cartesian multicategory of sets to any other $\Delta_R$-multicategory.

There is another application of this work which is more proof-theoretic. If $E\cup\{\phi\}$ is an $R$-theory, where $R$ is a balanced and modelable context structure. Then $E\vdash_R \phi$ if and only $E\vdash \phi$. In other words, one can potentially change the search space in the quest of finding a proof for $\phi$ from $E$ without losing all proofs, given that a proof exists. We also showed that for any modelable context structure $R$ and an $R$-theory $E$, the theory $E$ can linearized in a certain sense. We may construct multicategorical signature $\Sigma$ from $\sigma$ and $R$ and construct a linear $\Sigma$-theory $\overline{E}$ that has the property that $\overline{E}\vdash T_1\approx_{()} T_2$ if and only if $E\vdash_R \overline{T_1}_v\approx_v \overline{T_2}_v$ for all pure $\Sigma$–terms $T_1,T_2$ having the same type.

\section*{Acknowledgements}
This research was supported by a grant from the Fonds de la Recherche Scientifique - FNRS.
\bibliography{ref.bib}

\begin{thebibliography}{1}

\bibitem{gould2010coherence}
M.~R. Gould.
\newblock Coherence for categorified operadic theories, 2010.

\bibitem{FunctorialSemantics}
F.~W. Lawvere.
\newblock Functorial semantics of algebraic theories.
\newblock {\em Proceedings of the National Academy of Sciences of the United
  States of America}, 50(5):869--872, 1963.

\bibitem{Leinster_2004}
T.~Leinster.
\newblock {\em Higher Operads, Higher Categories}.
\newblock London Mathematical Society Lecture Note Series. Cambridge University
  Press, 2004.

\end{thebibliography}
\bibliographystyle{abbrv}
\end{document}